\documentclass[12pt, a4paper]{amsart}
\usepackage{amscd,amsmath,amssymb,amsthm,amsfonts}
\usepackage[alphabetic]{amsrefs}
\usepackage{mathrsfs}
\usepackage[shortlabels]{enumitem}
\usepackage[all]{xy}
\usepackage{xcolor}
\usepackage[hypertexnames=true, citecolor = green, colorlinks = false, linkcolor = red, linktoc = none, pdffitwindow = false, urlbordercolor = white]{hyperref}%

\def\ker{\operatorname{ker}}

\def\id{\operatorname{id}}

\def\Ad{\operatorname{Ad}}

\def\id{\operatorname{id}}

\def\N{\mathbb{N}}

\newcommand{\IC}[0]{\mathbb{C}}

 \newcommand{\IN}[0]{\mathbb{N}}

 \newcommand{\IT}[0]{\mathbb{T}}

 \newcommand{\IZ}[0]{\mathbb{Z}}

 \newcommand{\CB}[0]{\mathcal{B}}
 \renewcommand{\CD}[0]{\mathcal{D}}
 \newcommand{\CF}[0]{\mathcal{F}}
 \newcommand{\CH}[0]{\mathcal{H}}
 
\newcommand{\CK}[0]{\mathcal{K}} 
 
\newcommand{\CO}[0]{\mathcal{O}} \newcommand{\CP}[0]{\mathcal{P}}
\newcommand{\CQ}[0]{\mathcal{Q}} 
 
\newcommand{\CU}[0]{\mathcal{U}}

\newtheorem{thm}{Theorem}[section]
\newtheorem{corollary}[thm]{Corollary}
\newtheorem{lemma}[thm]{Lemma}

\newtheorem{proposition}[thm]{Proposition}

\theoremstyle{definition}
\newtheorem{definition}[thm]{Definition}

\theoremstyle{remark}
\newtheorem{remark}[thm]{Remark}

\newtheorem{example}[thm]{Example}

\numberwithin{equation}{section}

\begin{document}

\title{Permutative Representations of the $2$-adic ring $C^*$-algebra}

\author[V.~Aiello]{Valeriano Aiello}
\address{Department of Mathematics\\ Vanderbilt University\\ % 1362 Stevenson Center, Nashville, TN 37240, 
USA}
\email{valerianoaiello@gmail.com}

\author[R.~Conti]{Roberto Conti}
\address{Dipartimento di Scienze di Base e Applicate per l'Ingegneria \\ Sapienza Universit\`{a} di Roma \\ Italy}
\email{roberto.conti@sbai.uniroma1.it}

\author[S.~Rossi]{Stefano Rossi}
\address{Dipartimento di Matematica \\ Universit\`{a} di Roma Tor Vergata \\ Italy}
\email{rossis@mat.uniroma2.it}

\begin{abstract}
The notion of permutative representation is generalized to the $2$-adic ring $C^*$-algebra $\CQ_{2}$.  Permutative representations of $\CQ_2$ are then investigated
with a particular focus on the inclusion of the Cuntz algebra $\CO_2\subset\CQ_2$. Notably, every permutative representation of $\CO_2$ is shown to extend automatically to a permutative representation of $\CQ_2$ provided that an extension whatever 
exists. Moreover, all permutative extensions of a given representation of $\CO_2$ are proved to be unitarily equivalent to one another.  
Irreducible permutative representations of $\CQ_2$ are  classified in terms of irreducible permutative representations of the Cuntz algebra. Apart from 
the canonical representation of $\CQ_2$, every irreducible representation of $\CQ_2$ is the unique extension of an irreducible permutative representation of $\CO_2$. 
Furthermore, a permutative representation of $\CQ_2$ will decompose into a direct sum of irreducible permutative subrepresentations if and only if it restricts to $\CO_2$ as a \emph{regular} representation
in the sense of Bratteli-Jorgensen. 
As a result, a vast class of pure states of $\CO_2$ is shown to enjoy the unique pure extension property with respect to the inclusion $\CO_2\subset\CQ_2$.

%The unique pure extension property is also addressed.

\end{abstract}

\date{\today}
\maketitle
\tableofcontents

\section{Introduction}
It is common knowledge that  the Cuntz algebras display such a rich representation theory that any attempt at completely classifying even only their irreducible or factorial representations is bound to fail.
Even so, the study of their representations  has found striking applications in a wide range of seemingly distant fields such as
fractals, multiresolutions, self-similarity, symbolic dynamics, and wavelet theory to name but few \cite{J99,    DJ11, DJ, DP,  BJ97, Nek04, KW, KW06, J06, MSW07, AP09, MP11,  DJ14, DDHJ, Thom}. %ADD REFERENCESThe permutative
In addition, suitable classes of representations carefully selected to be studied thoroughly but still general enough to be of any interest at all do exist. 
Permutative representations are certainly a case in point. These are representations in which the generating isometries act in a rather simple way  mapping the vectors of a given orthonormal basis of a separable Hilbert space
to vectors of the same basis. In other terms, the action of the isometries is implemented by certain injective maps of $\IN$ into itself.  It is no wonder, therefore, that an in-depth analysis of such representations can
and in fact do give access to interesting connections with the theory of discrete dynamical systems as well as yielding examples where abstract aspects and notions from representation theory take concrete shape. Moreover, the general theory of permutative representations is very well understood, so much so
its description may safely be regarded as a more or less fully accomplished goal. Indeed, in their remarkable monograph \cite{BJ} Bratteli and Jorgensen provided an insightful analysis of all permutative representations, which is quite conclusive
when the so-called regular representations are dealt with. Notably, in the emerging picture the inclusion of the canonical UHF
subalgebra $\CF_2$ in $\CO_2$ is looked at as a kind of Gelfand pair.
A totally different approach which only slightly overlaps with part of the content of their analysis is taken in \cite{DavPitts}, where it is mostly von Neumann algebras 
to be focused on instead.  
The present work aims to extend this study to the $2$-adic ring $C^*$-algebra $\CQ_2$, which the present authors have already addressed in a couple of recent papers, \cite{ACR, ACR2}. The underlying idea is to exploit the far-reaching knowledge
accumulated over the years on the representation theory for Cuntz algebras to shed light on broader classes of
akin $C^*$-algebras, on whose representations very little is known so far. This is certainly the case with the $2$-adic ring $C^*$-algebra, that is  the universal $C^*$-algebra generated by a unitary $U$ and an isometry $S_2$ satisfying the two relations
$S_2U=U^2S_2$ and $S_2S_2^*+US_2S_2^*U^*=1$. Notably, it contains a copy of the Cuntz algebra $\CO_2$ in which the generating isometries are given
by $US_2$ and $S_2$. The rather explicit character of the inclusion $\CO_2\subset\CQ_2$ raises  a good many questions, some of which have been completely answered.
For instance, one may wonder whether an endomorphism of $\CO_2$ extends to an endomorphism of $\CQ_2$. To the best of our knowledge, it turns out that this is hardly ever the case, apart from the obvious examples, e.g.
inner automorphisms, gauge automorphisms, the flip-flop automorphism, and the canonical endomorphism, see \cite{ACR}. On the other hand, representations of $\CO_2$ are far more likely to extend to representations of $\CQ_2$
on the same Hilbert space. Indeed, it is a result by Larsen and Li, \cite{LarsenLi}, that a representation $\pi:\CO_2\rightarrow \CB(\CH)$ extends to a representation of $\CQ_2$
if and only if  the Wold unitaries of $\pi(S_1)$ and $\pi(S_2)$ are unitarily equivalent. In particular, all representations in which $\pi(S_1)$ and
$\pi(S_2)$ are both pure automatically extend. It is then natural to ask oneself how this  result specializes to permutative representations of $\CO_2$. In this regard, we prove without making use of the above theorem that a permutative representation
$\pi$ of $\CO_2$ extends to a permutative representation $\tilde\pi$ of $\CQ_2$, namely a representation in which the unitary $\tilde\pi(U)$ acts permuting the basis vectors, if and only if the injective maps $\sigma_1, \sigma_2:\IN\rightarrow \IN$ implementing the
isometries $\pi(S_1)$ and $\pi(S_2)$  have the same orbit structure when restricted to the invariant subsets $\bigcap_{n=1}^\infty \sigma_1^n(\IN)$ and $\bigcap_{n=1}^\infty\sigma_2^n(\IN)$ respectively. 
Still this does not answer the question of whether permutative extensions always exist as long as a possibly non-permutative extension exists. So we go on to show that not only is this the situation that actually occurs, but also that all permutative extensions are unitarily equivalent
to one another.  Furthermore, we discuss how many permutative extensions a permutative representation of $\CO_2$ may give rise to. We then shift our attention to irreducible permutative representations showing that these can be
described in quite a satisfactory fashion. Indeed, apart from the canonical representation of $\CQ_2$, every permutative representation of $\CQ_2$ restricts to $\CO_2$ as an irreducible representation of which it is the unique extension.
Conversely, every irreducible permutative representation of $\CO_2$ extends to a permutative representation of $\CQ_2$, apart from $\pi_+$ and $\pi_-$, which are the irreducible components of the restriction to $\CO_2$ of the canonical
representation of $\CQ_2$. Phrased differently, the irreducible permutative representations of the $2$-adic ring $C^*$-algebras are virtually
% essentially 
the same as those of the Cuntz algebra, all of which are well known, cf. \cite{Kaw06}. In addition, a permutative representation of $\CQ_2$ will decompose into the direct sum of irreducible permutative subrepresentations if and only its restriction to $\CO_2$ does, which is the same as asking that this restriction be regular in the sense of Bratteli-Jorgensen. In other words, the theory of permutative representations of $\CQ_2$ is thus entirely reconducted to and described in terms of that of $\CO_2$.
These findings further substantiate the idea that the inclusion $\CO_2\subset\CQ_2$ should be very rigid in more than one respect; for example, we had already  proved in \cite{ACR} that every endomorphism of
$\CQ_2$ that fixes $\CO_2$ pointwise must be trivial itself. 
In a similar spirit, we prove here as an application of our analysis that a great many pure states of the Cuntz algebra will  admit exactly one pure extension to the $2$-adic ring $C^*$-algebra.
Notably, all vector states associated with an irreducible permutative representation of $\CO_2$ can be extended to $\CQ_2$ in only one way.
\medskip 

In an effort to keep the text to a reasonable length we have preferred
to stick to the inclusion  $\CO_2 \subset \CQ_2$, although we do not see any
major obstacles to extending our analysis to the inclusions $\CO_n \subset
\CQ_n$, for any $n\in\IN$, as in the aforementioned references, cf. \cite{ACRS}. In particular, in \cite{ACRS} it is pointed out  that the canonical representation of $\CQ_n$ is still irreducible for every $n$, while its restriction to $\CO_n$ is not since it actually continues to decompose into
the direct sum of two irreducible components. More importantly, a theorem \`{a} la Larsen and Li is in fact still available for the inclusion $\CO_n\subset\CQ_n$ as well, which makes it possible to carry on a similar analysis. This, however, would entail
much more work to do, which is why we have resolved to go back to this generalization elsewhere. 
\medskip

The paper is organized in the following way.  
In Section \ref{prel}, after recalling the notation and collecting some basic results needed throughout the text,  we introduce  the notion of permutative representation of $\CQ_2$.
Section \ref{extending} is mainly focused on the problem of extending a permutative representation of $\CO_2$ to a permutative representation of $\CQ_2$. Notably, Theorem \ref{orbits} provides 
a necessary and sufficient condition for a permutative representation of $\CO_2$ to extend to a permutative representation of $\CQ_2$; Theorem \ref{uniext} proves that the extension is unique up to unitary equivalence, and 
Theorem \ref{multiplicity} shows that permutative extensions always exist as long as the representation extends.
Section \ref{general} highlights some general properties of permutative representations. In addition, Proposition \ref{intrcan} provides a characterization of the so-called canonical representation among
all permutative representations.
In Section \ref{decomposition} the analysis of permutative representation of $\CO_2$ by Bratteli-Jorgensen is reread in terms of $\CQ_2$.
Section \ref{converse} spells out some properties of the permutative representations of $\CO_2$ not to be easily found elsewhere in the literature.
In Section \ref{irreducible}, Theorem \ref{IrrRep} gives the list of all irreducible permutative representations of $\CQ_2$: apart from the canonical representation, these can all be obtained as the unique extension to $\CQ_2$ of an irreducible
representation of the Cuntz algebra; conversely, every such irreducible representation automatically
extends to $\CQ_2$.
General (i.e. possibly not irreducible) permutative representations
are dealt with in Proposition \ref{q2o2}, which shows the equivalence between complete reducibility in terms of permutative irreducibles and regularity of the restriction to $\CO_2$ in the sense of Bratteli-Jorgensen.
Section \ref{purestates} discusses the unique pure extension property for certain pure states of $\CO_2$ understood as a subalgebra of $\CQ_2$.
Finally, Section \ref{extendible} analyzes a class of permutative representations of $\CQ_2$ arising from permutative endomorphisms of $\CO_2$.

\section{Preliminaries and notation}\label{prel}
We start by recalling the definition of the $2$-adic ring $C^*$-algebra.

\begin{definition}
The $2$-adic ring $C^*$-algebra $\CQ_2$ is the \emph{universal} $C^*$-algebra generated by a unitary $U$
and a (proper) isometry $S_2$ such that $S_2U=U^2S_2$ and $S_2S_2^*+US_2S_2^*U^*=1$.
\end{definition}

 For a comprehensive account of its main properties the interested readers
are referred to \cite{ACR} and the references therein. Here we limit ourselves to reminding them that $\CQ_2$ is a simple, purely infinite $C^*$-algebra. In particular, its representations are  all faithful.
The so-called \emph{canonical representation}, which in this paper will be consistently denoted by $\rho_c$, plays a privileged role. This is the irreducible representation acting on the Hilbert space
$\ell^2(\IZ)$, with canonical orthonormal basis $\{e_k:k\in\IZ\}$, given by $\rho_c(U)e_k\doteq e_{k+1}$   and $\rho_c(S_2)e_k\doteq e_{2k}$, $k\in\IZ$. More details are again given in \cite{ACR}.   
Note, however, that $\rho_c(U)$ is a multiplicity-free unitary, that is $W^*(\rho_c(U))\subset\CB(\ell^2(\IZ))$ is a maximal abelian subalgebra. 
As already observed, the Cuntz algebra $\CO_2$, namely the universal $C^*$-algebra generated by two isometries $X_1, X_2$ such that
$X_1X_1^*+X_2X_2^*=1$, embeds into $\CQ_2$ through the injective $^*$-homomorphism that sends $X_1$ to $US_2$ and $X_2$ to $S_2$. The restriction of the canonical representation $\rho_c$ to $\CO_2$, which will henceforth be
denoted by $\pi_c$, is no longer irreducible. In fact, it decomposes into the direct sum of two inequivalent irreducible subrepresentations $\pi_+$ and $\pi_-$ given by the restriction
of $\pi_c$ to the invariant subspaces $\CH_+\doteq\overline{\textrm{span}}\{e_k:k\geq 0\}$ and $\CH_-\doteq\overline{\textrm{span}}\{e_k:k< 0\}$ respectively, see \cite{ACR}. 
Both $\pi_+$ and $\pi_-$ as well as $\pi_c$ are simple examples of permutative representations of $\CO_2$. In general, a representation $\pi:\CO_2\rightarrow \CB(\CH)$ is said to be
\emph{permutative} if there exists an orthonormal basis $\{e_n:n\in\IN\}$ such that $\pi(S_1)e_n=e_{\sigma_1(n)}$ and $\pi(S_2)e_n=e_{\sigma_2(n)}$, for every $n\in\IN=\{1,2,3,\ldots\}$, where
$\sigma_1, \sigma_2$ are two injections of $\IN$ into itself such that $\sigma_1(\IN)\cap\sigma_2(\IN)=\emptyset$ and $\IN=\sigma_1(\IN)\cup\sigma_2(\IN)$. A pair $(\sigma_1, \sigma_2)$ of injections of $\IN$ into
itself satisfying the above properties is sometimes referred to as a \emph{branching function system} of order $2$, cf. \cite{Kaw06}. Clearly,  the two properties correspond to the fact that the ranges of $\pi(S_1)$ and $\pi(S_2)$
decompose $\CH$ into their direct sum. It is also as clear that a similar definition can be given for representations of $\CQ_2$, which is what we do next.
\begin{definition}
A representation $\rho: \CQ_2\rightarrow \CB(\CH)$ on a separable Hilbert space is said to be permutative if there exists an orthonormal basis $\{e_n: n\in\IN\}$ of $\CH$ such that
 $\rho(S_2)e_n=e_{\sigma_2(n)}$ and $\rho(U)e_n=e_{\tau(n)}$, where $\sigma_2$ is an injection of $\N$ into itself and $\tau$ is a bijection of $\N$.
\end{definition}
It is plain that $\rho$ will be a representation if and only if $\sigma$ and $\tau$ satisfy $\sigma\circ\tau=\tau^2\circ\sigma$ and the pair $(\sigma, \tau\circ\sigma)$ is a branching function system of order $2$.
The canonical representation $\rho_c$ is obviously a permutative representation of $\CQ_2$.   It is quite obvious that a permutative representation of $\CQ_2$ restricts to a permutative representation of $\CO_2$  whose
branching function system is the pair $(\tau \circ \sigma_2,\sigma_2)$. It is not quite so obvious, though, that an extendible permutative representation of $\CO_2$ must also have permutative extensions. The
next section is entirely devoted to tackling this and other related problems. Here, instead, we collect some more or less known facts, whose proofs are nonetheless included for the sake of completeness.\\
  
For any injection $f:\IN\rightarrow\IN$ we denote by $S_f$ the isometry acting on $\ell^2(\IN)$ as $S_fe_k:=e_{f(k)}$, $k\in\IN$. Clearly, $S_f$ is a proper isometry if and only if $f$ is not surjective. 
More detailed information is given by the following results.

\begin{lemma}\label{isopower}
If $f$ is an injection of $\IN$ into itself, then the following set equality
$$\bigcap_{n=1}^\infty {\rm Ran}(S_f^n)=\overline{{\rm span}}\{e_j: j\in\bigcap_{n=1}^\infty f^n(\IN)\}$$
holds.
\end{lemma}

\begin{proof}

To begin with, we write down the following chain of equalities  
$${\rm Ran}(S_f^n)={\rm Ran}(S_{f^n})=\overline{{\rm span}}\{e_{f^n(i)}:i\in\IN\}=\{x\in\ell^2(\IN): (x, e_j)=0\,\,{\rm if}\, j\notin f^n(\IN) \}$$
which are all checked at once. Therefore, we also have the equalities 
$$\bigcap_{n=1}^\infty{\rm Ran}(S_f^n)=\{x\in\ell^2(\IN): (x, e_j)=0\,\,{\rm if}\, j\notin\cup_n f^n(\IN)\}=\bigcap_{n=1}^\infty\overline{{\rm span}}\{e_i: i\in\cap_n f^n(\IN)\}$$ 
\end{proof}
We recall that an isometry $S$ is said to be pure if $\bigcap_{n=1}^\infty \textrm{Ran}(S^n)=0$.
As a straightforward consequence of the previous lemma, we can also state the following result, which is singled out for convenience.
\begin{corollary} \label{cor-pureness}
The isometry $S_f$ is pure if and only if $\bigcap_{n=1}^\infty f^n(\IN)=\emptyset$.
\end{corollary}

Given an isometry $S$, the subspace $\bigcap_{n=1}^\infty \textrm{Ran}(S^n)$ may of course be non-trivial. However,  it will  always be invariant under the action of $S$, which in fact restricts to it 
as a unitary operator $W(S)$, known as the Wold unitary of $S$. Lemma \ref{isopower} also says that $W(S_f)$ is in fact a permutative unitary, since it operates as the restriction of $S_f$ to a closed subspace that is still
generated by a selected subset of the orthonormal basis one starts with.\\

There follows a handful of results all to do with the point spectrum $\sigma_p(S_f)$, by which we mean the (possibly empty) set of all eigenvalues of $S_f$.

\begin{proposition}
If $f: \IN\rightarrow \IN$ is an injection without periodic points, i.e. such that $f^{h}(k)\neq k$ for all $k, h\in\IN$, then $\sigma_p(S_f)=\emptyset$.
\end{proposition}
\begin{proof}
If we set $x\doteq\sum_j c_je_j$, the eigenvalue equation $S_fx=\lambda x$ reads
\begin{align*}
S_f (\sum_{j\in \IN} a_j e_j) & =\sum_{j\in \IN} a_j e_{f(j)}= \sum_{j\in \IN} \lambda a_j e_j
\end{align*}
for some $\lambda \in\IT$. 
%Obviously, the equation says that $c_j=0$ whenever $j\not\in f(\IN)$. 
%If $j$ does belong to $f(\IN)$, 
Obviously, the equation says that $|c_{j}| = |c_{f^{-1}(j)}|$ and by iterating it also says 
$$
|c_{j}| = |c_{f^{-k}(j)}| \qquad \forall k\in \IN %\textrm{as long as j\in}
$$
as long as $j\in\cap_{k\in \IN}f^k(\IN)$, otherwise we only have finite number equal coefficients.
This in turn implies $c_j$ must be zero: in the first case because the vector is in $\ell^2(\IN)$; in the second because $c_j=0$ if $j$ is not in  $f(\IN)$.
\end{proof}

\begin{proposition}
Let $f: \IN\rightarrow \IN$ be an injection with a unique fixed point and without other periodic points, then $\sigma_p(S_f)=\{1\}$.
\end{proposition}
\begin{proof}
Denote by $\bar{j}$ the fixed point.
If we set $x\doteq\sum_j c_je_j$, the eigenvalue equation $S_fx=\lambda x$ reads
\begin{align*}
S_f (\sum_{j\in \IN} c_j e_j) & =\sum_{j\in \IN} c_j e_{f(j)}= \sum_{j\in \IN} \lambda c_j e_j
\end{align*}
for some $\lambda \in\IT$. 
%Obviously, the equation says that $c_j=0$ whenever $j\not\in f(\IN)$. 
%If $j$ does belong to $f(\IN)$, 
Pick $j\neq \bar j$.
Obviously, the equation says that $|c_{j}| = |c_{f^{-1}(j)}|$ and by iterating it also says 
$$
|c_{j}| = |c_{f^{-k}(j)}| \qquad \forall k\in \IN
$$
as long as $j\in\cap_{k\in \IN}f^k(\IN)$, otherwise we only have finitely many equal coefficients. As in the previous lemma, this leads to $c_j=0$.
Therefore, we see that $v=c_{\bar j}e_{\bar j}$ and $\lambda=1$.
\end{proof}

In the following, by an orbit $O$ of any injective function $f:\IN\rightarrow\IN$ we will always mean a set of the form $O_{n_0}=\{f^n(n_0): n\in\IN\}\subset\IN$, where $n_0$ is a fixed natural number.
When $f$ is in addition surjective, as is the case with both $\tau$ and the restriction of $f$ to $\bigcap_{n=1}^\infty f^n(\IN)$  as long as this set is not empty,  then all orbits will always be assumed 
to be of the form $\{f^k(n_0):k\in\IZ\}$ unless otherwise stated.  
At any rate, it is clear that $\IN$ decomposes into the disjoint unions of orbits.
Finite orbits yield eigenvalues, as shown below. 

\begin{proposition}
Let $f: \IN\rightarrow \IN$ be an injection with only one 
finite orbit  $F\subset\IN$, %i.e. $F=f(F)$. %We denote this orbit by $F=\{ j_1, \ldots , j_n\}$. 
then $\sigma_p(S_f)=\sigma_p(S_f\upharpoonright_{V_F})$ where $V_F\doteq {\rm span}\{e_k\; | \; k\in F\}$. In particular, $\sigma_p(S_f)=\{z: z^n=1\}$, with $n=|F|$.
\end{proposition}
\begin{proof}
Let $ j_1, \ldots , j_n$ be the elements of $F$. 
If we set $x\doteq\sum_j c_je_j$, the eigenvalue equation $S_fx=\lambda x$ reads
\begin{align*}
S_f (\sum_{j\in \IN} a_j e_j) & =\sum_{j\in \IN} a_j e_{f(j)}= \sum_{j\in \IN} \lambda a_j e_j
\end{align*}
for some $\lambda \in\IT$. 
Pick $j\neq \bar j$.
Obviously, the equation says that $|c_{j}| = |c_{f^{-1}(j)}|$ and by iterating it also says 
$$
|c_{j}| = |c_{f^{-k}(j)}| \qquad \forall k\in \IN
$$
as long as $j\in\cap_{k\in \IN}f^k(\IN)$, otherwise we only have finitely many equal coefficients. As in the two previous lemmas, this leads to $c_j=0$.
Therefore, we see that $x=\sum_{h=1}^n c_{j_h}e_{j_h}$, that is $x\in V_F$.
\end{proof}

\begin{proposition}
Let $f: \IN\rightarrow \IN$ be an injective function with only 
 finitely many finite orbits, say % 
 $F_1, \ldots , F_k$. 
Then $\sigma_p(S_f)=\bigcup_{h=1}^k \sigma_p(S_f\upharpoonright_{F_h})=\bigcup_{h=1}^k\{z : z^{|F_h|}=1\}$.
\end{proposition}
\begin{proof}
Just apply the same arguments as above.
\end{proof}

It may be worth noting that any injection $\sigma:\IN\rightarrow\IN$ such that $\bigcap_n \sigma^n(\IN)$ is empty cannot have finite orbits.
However, the converse is by no means true, as shown by the next example.

%\begin{remark}
%Whenever $f:\IN\hookrightarrow\IN$ is not surjective, the finite $f$-invariant subsets are immediately seen to be at most finitely many, if any.
%\end{remark}

\begin{example} The map
$f: \IN\to\IN$ given by
$$
f(n)\doteq  \left\{ \begin{array}{ll}
1 & \textrm{ if $n=2$  }\\
n+2 & \textrm{ if $n$ is odd }\\
n-2 & \textrm{ if $n$ is even, $n\neq 2$} \\
\end{array} \right.
$$
is a bijection so $\bigcap_{n=1}^\infty f^n(\IN)=\IN$ and yet $f$ has no invariant subsets.
\end{example}
%At this point, it seems plausible to expect an $S_f$ to be pure if the associated $f:\IN\rightarrow \IN$ is an injection without finite invariant subsets. In fact, this turns out to be false. 
%A counterexample may be given as follows. Consider a partition $\N=A\cup B$, where $A$ and $B$ are both infinite. Choose a bijection $\Psi: A\rightarrow A$ devoid of finite invariant
%subsets, for instance the one exhibited in the above example, and any injection $\Phi:B\rightarrow B$ with no finite invariant subsets. By glueing $\Psi$ and $\Phi$ together, we get a non-surjective
%injection $f:\IN\rightarrow \IN$ with no finite invariant subsets. Nevertheless, it is straightforwardly seen that $\bigcap_n (S_f)^n(\ell^2(\N))=\overline{{\rm span}}\{e_j: j\in A\}$. To sum up, not only
%is $S_f$ not pure, but its Wold unitary may in a sense be taken as big as desired. 
\medskip

\section{Extending permutative representations from $\CO_2$ to $\CQ_2$}\label{extending}

We can now move on to the problem of deciding when a given permutative representation $\pi:\CO_2\rightarrow\CB(\CH)$ admits permutative extensions to $\CQ_2$.
We start by dealing with the case in which $\pi(S_1)$ and $\pi(S_2)$ are pure. In this situation, $\pi$ certainly extends and moreover its extension is unique. Even so, it is not obvious that
this extension is still permutative. This result, though, can be easily achieved by an application of the following lemma.

\begin{lemma}
Let $(\sigma_1,\sigma_2)$ be a branching function system of order $2$ such that the itersection $\bigcap_n \sigma_1^n(\IN)$
is empty. Then every $n\in\sigma_1(\IN)$ can be written as $n=\sigma_1^k(\sigma_2(m))$, where $k,m\in\IN$ are uniquely determined by $n$. 
\end{lemma}
\begin{proof}
If $n\in\sigma_1(\N)$, then $n=\sigma_1(i)$ for some $i\in\IN$. If now $i$ belongs to $\sigma_2(\IN)$, say $i=\sigma_2(j)$ for some $j\in\IN$, then $n=\sigma_1(\sigma_2(j))$, and we
are done. If it does not, then $i=\sigma_1(j)$ for some $j\in\IN$, and so $n=\sigma_1^2(j)$. We can now go on this way asking whether $j$ belongs to $\sigma_2(\IN)$ or not.
As the intersection $\bigcap \sigma_1^n(\IN)$ is empty, this procedure must come to an end in a finite number of steps, which means $n=\sigma_1^k(\sigma_2(m))$, with
$k\doteq{\rm min}\{l: n\notin \cap_{i=1}^l\sigma_1^i(\N)\}$. As for the unicity, suppose $\sigma_1^{k_1}(\sigma_2(m_1))=\sigma_1^{k_2}(\sigma_2(m_2))$. Without loss of generality, assume
$k_1>k_2$. By injectivity, we get $\sigma_1^{k_1-k_2}(\sigma_2(m_1))=\sigma_2(m_2)$. As the ranges of $\sigma_1$ and $\sigma_2$ are disjoint, we see that $k_1=k_2$.
Finally the injectivity of $\sigma_2$ gives $m_1=m_2$.  
\end{proof}
Now if we further assume that the itersection $\bigcap_{n=1}^\infty \sigma_2^n(\IN)$ is also empty, we uniquely recover a bijection $\tau:\IN\rightarrow\IN$ that satisfies the needed commutation rules with
$\sigma_1$ and $\sigma_2$. 
\begin{corollary}\label{pure}
Let $(\sigma_1,\sigma_2)$ be a  branching function system of order $2$ as above. If the intersection $\bigcap_{n=1}^\infty \sigma_2^n(\IN)$ is empty as well, then there exists a unique bijection 
$\tau$ of $\IN$ such that  $\tau\circ\sigma_2=\sigma_1$ and  $\tau\circ\sigma_1=\sigma_2\circ\tau$.
\end{corollary}
\begin{proof}
The $\tau$ in the statement can be described explicitly by setting $\tau(\sigma_2(n))\doteq \sigma_1(n)$ for every $n\in\IN$ and $\tau(\sigma_1^k(\sigma_2(m)))\doteq \sigma_2^k(\sigma_1(m))$, for every $k>0$ and $m\in\IN$. The proposition above just guarantees that $\tau$ is defined everywhere.  Let us check that the condition $\tau\circ\sigma_1=\sigma_2\circ\tau$ holds. 
We observe that by the previous proposition we have that
\begin{align*}
\tau\circ \sigma_1(n) &  =\tau\circ \sigma_1^k \circ \sigma_2(m) = \sigma_2^k \circ \sigma_1(m)
\end{align*}
for some $k, m, n \in\IN$. Now there are two case we have to deal with, one for $k=1$ and one for $k\geq 2$. 
In the first case, the relation $\sigma_1(n)=\sigma_1\circ\sigma_2(m)$ implies that $n=\sigma_2(m)$. The claim then follows by noticing that 
$$
\sigma_2\circ\tau(n) =\sigma_2\circ\tau \circ \sigma_2(m) = \sigma_2 \circ \sigma_1(m)\; .
$$
%In the first case, one the one hand we have that $\tau\circ \sigma_1(n) =\sigma_2 \circ \sigma_1(m)$ and the equation $\tau\circ\sigma_1(n)=\sigma_2\circ\tau(n)$ holds if and only if $\sigma_1(m)=\tau(n)$. Now, $\tau(n)=\tau\circ\sigma_1\circ\sigma_2(m)=\sigma_2\circ\sigma_1(m)$
%\begin{align*}
%\sigma_2\circ \tau(n) & =\sigma_2\circ \tau\circ \sigma_1^{k-1} \circ \sigma_2(m) = \sigma_2\circ \sigma_2^{k-1} \circ \sigma_1(m) =  \sigma_2^k \circ \sigma_1(m)
%\end{align*}
In the second case, we see that by the very definition of $\tau$ we have that 
\begin{align*}
\sigma_2\circ \tau(n) & =\sigma_2\circ \tau\circ \sigma_1^{k-1} \circ \sigma_2(m) = \sigma_2\circ \sigma_2^{k-1} \circ \sigma_1(m) =  \sigma_2^k \circ \sigma_1(m)
\end{align*}
and we are done. 
\end{proof}

The corollary above can now  be reinterpreted in terms of permutative representations.
\begin{proposition}\label{pureperm}
Let $\pi$ be a permutative representation of $\CO_2$ such that $\pi(S_1)$ and $\pi(S_2)$ are both pure. Then the necessarily unique extension $\tilde\pi$ of $\pi$ to $\CQ_2$ is a permutative representation of $\CQ_2$.
\end{proposition}
\begin{proof}
Indeed, $\tilde\pi(U)$ is implemented by the bijection $\tau:\IN\rightarrow\IN$ produced in Corollary \ref{pureperm}.
\end{proof}

When the isometries are not pure, a much more intriguing picture comes out. The resulting situation is fully covered by the next theorem.

\begin{thm}\label{orbits}
Let $\pi:\CO_2\rightarrow\CB(\CH)$ be a permutative representation induced by 
the  branching function system $(\sigma_1,\sigma_2)$. Then the following conditions are equivalent:
\begin{enumerate}
\item the permutative representation $\pi$ of $\CO_2$ extends to a permutative representation $\widetilde \pi$ of $\CQ_2$;
\item if $Orb_1=\{F^1_i: i\in I\}$ and $Orb_2=\{F^2_j: j\in J\}$ are the sets of all orbits of $\sigma_1\upharpoonright_{\cap_{n\geq 1}\sigma_1^n(\IN)}$ and $\sigma_2\upharpoonright_{\cap_{n\geq 1}\sigma_2^n(\IN)}$  respectively, there exists a bijection $\Psi: I\rightarrow J$ such that  $|F^1_i|=|F^2_{\Psi(i)}|$ for every $i\in I$.
\end{enumerate}
In either case $\widetilde\pi(U)$ always satisfies  $\widetilde\pi(U) e_{\sigma_2(i)}=e_{\sigma_1(i)}$ $i\in\IN$ and
$\widetilde\pi(U) e_{\sigma_1^k\sigma_2(i)}=e_{\sigma_2^k\sigma_1(i)}$ $i\in\IN$, $k\in\IN$. Assuming for simplicity $I=J$ and $\Psi =\id$, $F_j^1=\{\sigma_1^k(n_j)\; |\; k\in\IZ\}$, $F_j^2=
\{\sigma_2^k(m_j)\; |\; k\in\IZ\}$,
 all permutative extensions are obtained by the formulas 
$\widetilde\pi(U) e_{\sigma_1^k(n_j)}=e_{\sigma_2^{k+l}(m_{j})}$ for $l\in\IN$.
\end{thm}

\begin{proof}
The implication $1\Rightarrow 2$ is easily proved, for a bijective correspondence between even all orbits of $\sigma_1$ and $\sigma_2$  is provided by
%$$\Phi(\{\sigma_1^k(n_0):k\in\IN\})\doteq \{\sigma_2^k(\tau^{-1}(n_0)): k\in\IN\}\quad \textrm{for any $n_0\in\IN$}$$
$$\Phi(\{\sigma_1^k(n_0)\}_k)\doteq \{\sigma_2^k(\tau^{-1}(n_0))\}_k$$
if $\tau$ is a bijection of $\IN$ such that $\sigma_1=\tau\circ\sigma_2$ and $\sigma_2\circ\tau=\tau\circ\sigma_1$. That $\Phi$ is actually a bijection is  immediately seen, as is
the equality $|\Phi(\{\sigma_1^k(n_0)%:k\in\IN
\}_k)|=|\{\sigma_2^k(\tau^{-1}(n_0))%:k\in\IN
\}_k|$.\\

The implication $2\Rightarrow 1$ requires slightly harder work. We need to determine a permutative unitary $\widetilde\pi(U)$ that satisfies
the two defining relations of $\CQ_2$, which rewrite as  $\tau\circ\sigma_2=\sigma_1$ and $\tau\circ\sigma_1^k\circ\sigma_2=\sigma_2^k\circ\sigma_1$, $k\in\IN$, if
$\tau$ is a bijection of $\N$ implementing $\widetilde\pi(U)$, i.e. $\widetilde\pi(U)e_k=e_{\tau(k)}$ for every $k\in\IN$. In other terms, the two relations uniquely 
determine what $\tau$ must be on $\bigcup_{n=0}^\infty\sigma_1^n\circ\sigma_2(\IN)$. This is where our hypothesis starts playing its role, for $\IN\setminus \bigcup_{n=0}^\infty\sigma_1^n\circ\sigma_2(\IN)$ is just $\bigcap_{n=1}^\infty \sigma_1^n(\IN)$. But $\bigcap_{n=1}^\infty \sigma_1^n(\IN)$ decomposes into the disjoint union of $\sigma_1$-orbits by means of a standard application of Zorn's lemma. Now we can define $\tau$ on any of such orbits $\{\sigma_1^k(n_0):k\in\IN\}$ by setting
$\tau(\sigma_1^k(n_0))\doteq \sigma_2^{k+l}(m_0)$, where $l$ is any integer number, and $\{\sigma_2^k(m_0): k\in\IZ\}$ is the corresponding $\sigma_2$-orbit.
By doing so, we obviously get a now everywhere defined map $\tau$, that is bijective by definition. We are thus left with the task of checking that $\sigma_2\circ\tau=\tau\circ\sigma_1$ continues 
to hold on the whole $\IN$. This in turn is shown by the following computation 
$$
\sigma_2(\tau(\sigma_1^k(n_0)))=\sigma_2(\sigma_2^{k+l}(m_0))=\sigma_2^{1+k+l}(m_0)=\sigma_2^{k+1+l}(m_0)=\tau(\sigma_1^{k+1}(n_0))=\tau(\sigma_1(\sigma_1^k(n_0)))
$$
\end{proof}

Interestingly, the theorem also gives full information to reckon how many permutative extensions there can be. Firstly, if the restriction of $\sigma_i$ to $\bigcap_{n=1}^\infty\sigma_i^n(\IN)$ consists only
of finitely many finite orbits, then the extensions will be finitely many as well.
% In particular, the case of $\bigcap_n\sigma_i^n(\IN)$ being  reduced to fixed points yields a unique extension.
Secondly, if there are finitely many infinite orbits, the extentions will be countably many. Lastly, if there are infinitely many infinite orbits,  the extensions will be uncountably many.
Irrespective of how many permutative extensions exist, they turn out to be all unitarily equivalent to one another, as we prove next.

\begin{thm} \label{uniext}
If $\pi: \CO_2\to \CB(H)$ is an extendible permutative representation, then all its permutative extensions (w.r.t. the same orthonormal basis in which $\pi$ is permutative)  are unitarily equivalent.
%Then there exists a unique permutative extension (with respect to the same orthonormal basis) to $\CQ_2$   up to unitary equivalence. 
\end{thm}
\begin{proof}
Let $\tilde\pi _1: \CQ_2\to \CB(H)$ and $\tilde\pi _2: \CQ_2\to \CB(H)$ be two permutative representations extending $\pi$. We will show that there exists a unitary $V: H\to H$ such that $\Ad(V)\circ \tilde\pi_1 = \tilde\pi_2$. As usual we denote by $\{e_k\}_{k\in\IN}$ the orthonormal basis and by $\{\sigma_1, \sigma_2\}$ the branching function system of order $2$. Consider the following subspace of $H$
\begin{align*}
%& M_1\doteq\bigcap_{n=1}^\infty{\rm Ran}(\pi(S_1)^n)\\
& M_2\doteq\bigcap_{n=1}^\infty{\rm Ran}(\pi(S_2)^n)
\end{align*} 
Suppose that $\widetilde\pi_1(U) e_{\sigma_1^k(n_j)}=e_{\sigma_2^{k}(m_{j'})}$ and $\widetilde\pi_2(U) e_{\sigma_1^k(n_j)}=e_{\sigma_2^{k}(m_{j''})}$ (we're using the notation of Theorem \ref{orbits}).
We define  $V\upharpoonright_{M_2}: M_2\to M_2$ as $V e_{\sigma_2^k(m_{j'})} \doteq e_{\sigma_2^k(m_{j''})}$ and we set %$V\upharpoonright_{M_2^\bot}: M_2^\bot \to M_2^\bot$ as 
$V\upharpoonright_{M_2^\bot} = \id\upharpoonright_{M_2^\bot}$.

Clearly $\Ad(V)\circ \tilde\pi_1 (S_2)\upharpoonright_{M_2^\bot}= \tilde\pi_2(S_2)\upharpoonright_{M_2^\bot}$ because $\pi (S_2)(M_2^\bot)\subset M_2^\bot$ .
The following computation shows that $\Ad(V)\circ \tilde\pi_1 (S_2)\upharpoonright_{M_2}= \tilde\pi_2(S_2)\upharpoonright_{M_2}$
\begin{align*}
V\tilde\pi_1(S_2)V^* e_{\sigma_2^k(m_{j''})} & = V\tilde\pi_1(S_2)e_{\sigma_2^k(m_{j'})}= V e_{\sigma_2^{k+1}(m_{j'})} = e_{\sigma_2^{k+1}(m_{j''})} = \tilde\pi_2(S_2) e_{\sigma_2^k(m_{j''})} 
\end{align*}
Now it is enough to check that $\Ad(V)\circ \tilde\pi_1(U^*) = \tilde\pi_2(U^*)$. The following computation shows that $\Ad(V)\circ \tilde\pi_1(U^*) \upharpoonright_{M_2}= \tilde\pi_2(U^*)\upharpoonright_{M_2}$
%\begin{align*}
%V\tilde\pi_1(U^*)V^* e_{\sigma_2^k(m_{j''})} & = V\tilde\pi_1(U^*)e_{\sigma_2^k(m_{j'})}= V e_{\sigma_1^{k}(m_{j'})} = e_{\sigma_1^{k}(m_{j''})} = \tilde\pi_2(U^*) e_{\sigma_2^k(m_{j''})} 
%\end{align*}
\begin{align*}
V\tilde\pi_1(U^*)V^* e_{\sigma_2^k(m_{j''})} & = V\tilde\pi_1(U^*)e_{\sigma_2^k(m_{j'})}= V e_{\sigma_1^{k}(n_{j})} = e_{\sigma_1^{k}(n_{j})} = \tilde\pi_2(U^*) e_{\sigma_2^k(m_{j''})} 
\end{align*}
%For the equality $\Ad(V)\circ \tilde\pi_1(U^*) \upharpoonright_{M_2^\bot}= \tilde\pi_2(U^*)\upharpoonright_{M_2^\bot}$ %follows from the fact that $\tilde\pi_2(U^*) \upharpoonright_{M_2^\bot}\subset M_2^\bot$. Indeed, the last claim follows by observing that the subspace $M_2^\bot$ is spanned by the elements of the vectors of the form $e_{\sigma_1(i)}$ and
We observe that $\Ad(V)\circ \tilde\pi_1(U^*) e_{\sigma_2^k(\sigma_1(i))}=\tilde\pi_2(U^*) e_{\sigma_2^k(\sigma_1(i))}$ for $i, k\in\IN$ and thus $\Ad(V)\circ \tilde\pi_1(U^*) \upharpoonright_{S_2^kS_1(H)}= \tilde\pi_2(U^*)\upharpoonright_{S_2^kS_1(H)}$, which in turn implies 
$$
\Ad(V)\circ \tilde\pi_1(U^*) \upharpoonright_{S_2(H)}= \tilde\pi_2(U^*)\upharpoonright_{S_2(H)}\; .
$$
We observe that $\Ad(V)\circ \tilde\pi_1(U^*) e_{\sigma_1^k(\sigma_2(i))}=\tilde\pi_2(U^*) e_{\sigma_1^k(\sigma_2(i))}$ for $i, k\in\IN$ and thus $\Ad(V)\circ \tilde\pi_1(U^*) \upharpoonright_{S_1^kS_2(H)}= \tilde\pi_2(U^*)\upharpoonright_{S_1^kS_2(H)}$. We also have that 
\begin{align*}
\Ad(V)\circ \tilde\pi_1(U^*) e_{\sigma_1(n_j)} & =V\tilde\pi_1(U^*) V^*e_{\sigma_1(n_j)}=V\tilde\pi_1(U^*) e_{\sigma_1(n_j)} = V  e_{\sigma_2(n_j)}\\
& =  e_{\sigma_2(n_j)} = \tilde\pi_1(U^*) e_{\sigma_1(n_j)}
\end{align*}
where $n_j\in \cap_k \sigma_1^k(\IN)$. This shows that $\Ad(V)\circ \tilde\pi_1(U^*) \upharpoonright_{S_1(H)}= \tilde\pi_2(U^*)\upharpoonright_{S_1(H)}$ and we are done.%The equality $\Ad(V)\circ \tilde\pi_1(U^*) \upharpoonright_{M_2^\bot}= \tilde\pi_2(U^*)\upharpoonright_{M_2^\bot}$ follows from the following computations
\end{proof}

However conclusive, Theorem \ref{orbits} still leaves something to be desired, insofar as it does not exclude the possibility that an extendible permutative representation of $\CO_2$ may have no
permutative extensions. This gap is bridged by  the theorem below, which says such a situation in fact never occurs.

\begin{thm}\label{multiplicity}
If a permutative representation $\pi$ of $\CO_2$ 
%such that either $\bigcap_n\pi(S_1)^n$ or $\bigcap_n\pi(S_2)^n$ is finite dimensional 
extends to a representation of $\CQ_2$, then it also admits a permutative extension.
\end{thm}
\begin{proof}
Let us set $M_i\doteq\bigcap_{n=1}^\infty{\rm Ran}(\pi(S_i)^n)$ and $W(S_i)\doteq \pi(S_i)\upharpoonright_{\cap_n\pi(S_i)^n}$,
$i=1,2$. In view of Theorem \ref{orbits}, all we have to make sure is that  $W(S_1)$ and $W(S_2)$  are induced by  permutations having the same orbit structure. What we already know is that $W(S_1)$ and $W(S_2)$ are unitarily equivalent. 
Therefore, it all boils down to proving that if two permutative unitaries are unitarily equivalent, then they are induced by orbit-equivalent bijections of $\IN$. This should be a known fact from multiplicity theory. Even so, we do include an argument
for want of an exhaustive reference. Let $\sigma$ be a bijection of $\N$ and let $U_\sigma$ be the corresponding permutative unitary. The decomposition of $\IN$ into the disjoint union of $\sigma$-orbits, say $\IN=\bigsqcup_{i\in I} O_i$, yields
a decomposition of $\ell^2(\N)$ into a direct sum of $U_\sigma$-cyclic subspaces. More explicitly, we have $\ell^2(\IN)=\bigoplus_{i\in I} H_{O_i}$, where $H_{O_i}\doteq\overline{{\rm span}} \{e_l: l\in O_i\}$.
Now if an orbit $O_i$ is finite with $|O_i|=k$, then $U_\sigma\upharpoonright_{H_{O_i}}$ is just (up to unitary equivalence) the $k$ by $k$ permutative matrix corresponding to the cycle of length $k$, which we denote by $V_k$. 
If it is infinite, then 
$U_\sigma\upharpoonright_{H_{O_i}}$ is (up to unitary equivalence) the operator $V_\infty$ acting on $L^2(\IT,\mu_{\rm Leb} )$ by multiplying by $z$, that is $(V_\infty f)(z)\doteq zf(z)$ $\mu_{\rm Leb}$ a.e.
Our unitary $U_\sigma$ is thus seen to decompose into the direct sum of multiplicity-free components as $U_\sigma=\bigoplus_{i=1}^\infty n_i V_i\oplus n_\infty V_\infty$, where $n_i$ is an integer (possibly zero or infinite). 
As is well known, this decomposition is unique and depends only on the unitary equivalence class. In other words, if $U_1=\oplus_i n_iV_i\oplus n_\infty V_\infty$ and $U_2=\oplus_i m_iV_i\oplus m_\infty V_\infty$, then 
$U_1$ and $U_2$ are unitarily equivalent if and only if $n_i=m_i$, for every $i\in\IN$ and $n_\infty=m_\infty$. 
Indeed, the finite dimensional components are easily dealt with, that is $U_1\cong U_2$ immediately leads to $n_i=m_i$ for every $i\in \IN$, since any
unitary equivalence between $U_1$ and $U_2$ must preserve the finite dimensional blocks. This in turn
implies $n_\infty V_\infty\cong m_\infty V_\infty$, which is possible only if $n_\infty=m_\infty$, see e.g. 
\cite{Davidson}. Phrased differently, it is clear that the unitary that intertwines $U_1$ and $U_2$ can now be assumed  to be a permutative unitary. 
\end{proof}

\begin{remark}\label{charpol}
When either $\bigcap_{n=1}^\infty \pi(S_1)^n$ or $\bigcap_{n=1}^\infty \pi(S_2)^n$ is finite dimensional we need hardly bother such an advanced instrument as multiplicity theory, for we have the characteristic polynomial at our disposal.
Indeed, if $U\in M_n(\IC)$ is a permutation unitary, i.e.  $Ue_i=e_{\sigma(i)}$, $i=1,2,\ldots, n$,  with $\{e_1,e_2,\ldots, e_n\}\subset \IC^n$ being a suitable orthonormal basis and $\sigma$ a permutation of $\{1,2,\ldots n\}$, then its 
characteristic polynomial $p_U$, which is obviously a unitary invariant, uniquely factorizes as 
$$
p_U(\lambda)=(\lambda^{n_1}-1)(\lambda^{n_2}-1)\ldots(\lambda^{n_k}-1)
$$ 
with  $n_1+n_2+\ldots+n_k=n$, where $k$ is the number of orbits of $\sigma$ and $n_i$ is the cardinality of the $i$-th orbit.
\end{remark}

\section{General properties of permutative representations of $\CQ_2$} \label{general}
We ended the previous section by remarking that in a permutative representation $\rho:\CQ_2\rightarrow\CB(\CH)$ induced by the pair $(\sigma_2,\tau)$ the bijection 
$\tau:\IN\rightarrow\IN$ can only have (at most countably many) infinite orbits. In order to prove this, we
need to make use of the following very simple result, which is nevertheless given a statement to itself for convenience. 
Before doing that, we take the opportunity to recall some very standard notation, of which we will have to make intensive use as of now to ease the computations.  We donote by
$W_2^k$ the set of multi-indices of length $k$ in the alphabet $\{1,2\}$ and by $W_2$  the set of all multi-indices of any finite length, i.e.
$W_2\doteq \bigcup_{k=1}^\infty W_2^k$. For any given $\alpha=(\alpha_1,\alpha_2,\ldots,\alpha_n)\in W_2$, we denote the monomial $S_{\alpha_1}S_{\alpha_2}\ldots S_{\alpha_n}\in \CO_2$ by $S_\alpha$. 
The diagonal projection $P_\alpha$ is by definition $S_\alpha S_\alpha^*$. Finally, the $C^*$-subalgebra generated by all projections $P_\alpha$ is denoted by $\CD_2$ and is often referred to as the diagonal
subalgebra of $\CO_2$. It is well known that $\CD_2$ is a maximal abelian subalgebra of the Cuntz algebra $\CO_2$.

\begin{lemma} \label{lemma-orth-proj}
The projections  $\Ad(U^h)(P_\alpha)$ and  $P_\alpha$ are orthogonal for any $h\in\IN$ and $\alpha\in W_2$ such that $0<h<2^{|\alpha |}$.

\end{lemma}
\begin{proof}
In the canonical representation the range of $P_\alpha$ is easily seen to be the subspace of $\ell^2(\mathbb{Z})$ generated by $\{ e_{2^{|\alpha |}k+l} \: | \: k\in \mathbb{Z} \}$ for a certain $l\in \mathbb{N}$, see e.g. the discussion before Lemma 6.22 in \cite{ACR}. The claim now follows at once. 
\end{proof}

\begin{thm}\label{noperiodicpoints}
The bijection $\tau:\IN\rightarrow\IN$ implementing $\rho(U)$ in a permutative representation $\rho: \CQ_2\to \CB(\CH)$ has no periodic points, i.e. for any given $n_0\in\IN$ one has $\tau^n(n_0)\neq n_0$, $n\in\IN$. 
\end{thm}
\begin{proof}
Without loss of generality we may suppose that $\rho(U)^ne_1=e_1$. We have to consider two cases separately: $e_1=\rho(S_1)e_i$ and $e_1=\rho(S_2)e_i$.
%$n$ is even and $n$ odd.  
We start with the case $e_1=\rho(S_1)e_i$ and suppose that $n=2k$.  We have that
$$
\rho(S_1)e_i=e_1 = \rho(U^n)e_1= \rho(U)^n \rho(S_1)e_i=\rho(U)^{2k} \rho(S_1)e_i= \rho(S_1)\rho(U)^{k} e_i
$$ 
which implies that $e_i=\rho(U)^ke_i$. Now choose a projection $\rho(S_\alpha S_\alpha^*)$ such that $\rho(S_\alpha S_\alpha^*)e_i=e_i$, $|\alpha |\geq k+1$. By Lemma \ref{lemma-orth-proj} we have that  $\rho(S_\alpha S_\alpha^*)+\rho(U^k)^*\rho(S_\alpha S_\alpha^*)\rho(U)^k\leq 1$. By using the orthogonality of the projections we see that
\begin{align*}
& \| \rho(S_\alpha S_\alpha^*)e_i+\rho(U^k)^*\rho(S_\alpha S_\alpha^*)\rho(U)^k e_i\|^2=\\
& = \| \rho(S_\alpha S_\alpha^*)e_i+\rho(U^k)^*\rho(S_\alpha S_\alpha^*)e_i\|^2=\\
& = \| \rho(S_\alpha S_\alpha^*)e_i\|^2+\|\rho(U^k)^*\rho(S_\alpha S_\alpha^*)e_i\|^2=2
\end{align*}
which is absurd.\\
Now we deal with the case when $n=2k+1$ (and still $e_1=\rho(S_1)e_i$). We have that
$$
e_1 = \rho(U^n)e_1= \rho(U^n) \rho(S_1)e_i=\rho(U)\rho(U)^{2k} \rho(S_1)e_i= \rho(U)\rho(S_1)\rho(U)^{k} e_i = \rho(S_2)\rho(U)^{k+1} e_i
$$ 
So we see that $e_1=\rho(S_1)e_i=\rho(S_2)\rho(U)^{k+1} e_i$, which is absurd because the two generating isometries have orthogonal ranges.

Now we deal with the case $e_1=\rho(S_2)e_i$. Suppose that $n=2k$.  We have that
$$
\rho(S_2)e_i=e_1 = \rho(U)^ne_1= \rho(U)^n \rho(S_2)e_i=\rho(U)^{2k} \rho(S_2)e_i= \rho(S_2)\rho(U)^{k} e_i
$$ 
which implies that $e_i=\rho(U)^ke_i$. Now choose a projection $\rho(S_\alpha S_\alpha^*)$ such that $\rho(S_\alpha S_\alpha^*)e_i=e_i$, $|\alpha |\geq k+1$. Clearly, we have that  $\rho(S_\alpha S_\alpha^*)+\rho(U^k)^*\rho(S_\alpha S_\alpha^*)\rho(U)^k\leq 1$. By using the orthogonality of the projections we see that
\begin{align*}
& \| \rho(S_\alpha S_\alpha^*)e_i+\rho(U^k)^*\rho(S_\alpha S_\alpha^*)\rho(U)^k e_i\|^2=\\
& = \| \rho(S_\alpha S_\alpha^*)e_i+\rho(U^k)^*\rho(S_\alpha S_\alpha^*)e_i\|^2=\\
& = \| \rho(S_\alpha S_\alpha^*)e_i\|^2+\|\rho(U^k)^*\rho(S_\alpha S_\alpha^*)e_i\|^2=2
\end{align*}
which is absurd.\\
Now we deal with the case when $n=2k+1$ (and still $e_1=\rho(S_2)e_i$). We have that
$$
e_1 = \rho(U^n)e_1= \rho(U^n) \rho(S_2)e_i=\rho(U)\rho(U)^{2k} \rho(S_2)e_i= \rho(U)\rho(S_2)\rho(U)^{k} e_i = \rho(S_1)\rho(U)^{k} e_i
$$ 
So we see that $e_1=\rho(S_2)e_i=\rho(S_1)\rho(U)^{k} e_i$, which is absurd because the two generating isometries have orthogonal ranges.
\end{proof}

A standard application of the above theorem also yields the following result. 
\begin{corollary}\label{intrper}
Let $\rho: \CQ_2\to \CB(\CH)$ be a permutative representation. Then $\rho(U)$ has no eigenvectors. 
\end{corollary}
\begin{proof}
Suppose on the contrary there does exist a non-zero $x\in\ell^2(\IN)$ such that $\rho(U)x=\lambda x$, for some $\lambda\in\IT$. With
$x=\sum_k c_k e_k$ the eigenvalue equation reads as $\sum_k c_ke_{\tau(k)}=\sum_k \lambda c_ke_k$, whence  $|c_{k_0}|=|c_{\tau^{-1}(k_0)}|=|c_{\tau^{-2}(k_0)}|=\ldots=
|c_{\tau^{-n}(k_0)}|$ for every $n$. However, if $c_{k_0}$ is any non-zero coefficient of $x$ the equalities found above say $x$ is not in $\ell^2(\IN)$, as the foregoing proposition
says it has infinitely many coefficients with the same non-zero absolute value.
\end{proof}

Given a representation $\rho$ of $\CQ_2$ one might wonder whether there is a general method for ruling out the possibility that $\rho$ may ever act as a permutative representation with respect to a certain orthonormal basis.  
In fact, answering such a question might prove to be a demanding task. Even so, Corollary \ref{intrper} does offer a simple, albeit limited, answer to the problem. Indeed, any representation
$\rho:\CQ_2\rightarrow\CB(\CH)$ in which the point spectrum of $\rho(U)$ is not empty cannot be permutative with respect to any orthonormal basis.  Notably, the so-called \emph{interval picture} of $\CQ_2$, see
\cite{ACRS}, is an example of such a representation.  In fact, the interval picture is not even permutative at the level of the Cuntz algebra. 

\begin{corollary}
The interval picture $\pi: \CO_2 \to\CB(L^2([0,1]))$  is not a  permutative representation with respect to any orthonormal basis.
\end{corollary}
\begin{proof}
As is known, see e.g. \cite{ACRS}, $\pi(S_1)$ and $\pi(S_2)$ are pure, hence there is only one extension $\tilde\pi$ to $\CQ_2$, which is permutative if and only if
$\pi$ is.  Now the operator $\tilde\pi(U)$ does have an eigenvector (the constant function $1$), which means $\pi$ is not permutative.  
\end{proof}
\begin{remark}
It is worth noting that the interval picture of $\CO_2$ can also be realized as the GNS representation of the Cuntz state associated with the vector $(1/\sqrt{2}, 1/\sqrt{2})\in\IC\oplus \IC$.
\end{remark}
The number of the (infinite) orbits of $\tau$ is certainly an invariant for a given permutative representation $\rho$, which coincides with the multiplicity of $\rho(U)$.
However, it is quite a weak invariant, which can by no means give a complete classification of all representations of $\CQ_2$. Indeed, there are uncountably many irreducible permutative representations, as we shall show afterwords,
whereas the values assumed by our invariant range in a countable set.  Even so, the invariant  does become unexpectedly fine when $\tau$ has only one orbit, in which case the corresponding representation must be
the canonical representation.  

\begin{proposition}\label{intrcan}
Let $\rho$ be a permutative representation of $\CQ_2$ in which the bijection $\tau$ implementing $\rho(U)$ has only one orbit, that is
$$\{\tau^k(0) \ : \ k \in {\IZ}\} = \IN \ . $$  Then $\rho$ is unitarily equivalent to
the canonical representation $\rho_c$. In particular, $\rho$ is also irreducible.
\end{proposition}
\begin{proof}
By hypothesis we can write $\rho(U)e_n=e_{\tau(n)}$, $n\in\IN$. Furthermore, the set equality $\IN=\{\tau^k(0): k\in\IZ\}$ allows us to give an explicit bijection
$\Psi:\IN\rightarrow\IZ$ by setting $\Psi(n)\doteq k$ if $\tau^k(0)=n$. We now claim that $\Psi(\tau(\Psi^{-1}(k)))=k+1$, for every $k\in\IZ$. Indeed, 
% if we define $n$ to be $\Psi^{-1}(k)$ then $\tau^k(0)=n$. 
by definition, $\Psi^{-1}(k) = \tau^k(0)$.
Therefore, we also have $\tau(\Psi^{-1}(k))=\tau^{k+1}(0)$. The last equality  finally reads as $\Psi(\tau(\Psi^{-1}(k)))=k+1$. 
Denoting by $V$ the unitary from $\ell^2(\IN)$ onto  $\ell^2(\IZ)$ that sends $e_n$ to $e_{\Psi(n)}$, the equality 
$V\rho(U)V^* e_k=e_{k+1}=\rho_c(U)e_k$ is immediately seen to hold for every $k\in\IZ$. Now $V\rho(S_2)V^*e_k=e_{\sigma(k)}$, for a suitable injection $\sigma$ of $\IZ$ into itself.
In terms of the maps $\tau$ and $\sigma$ the commutation relation $S_2U=U^2S_2$ may be rewritten as $\sigma(k+1)=\sigma(k)+2$, $k\in\IZ$ (just apply $\Ad(V) \circ \rho$ to the defining relation).
All the solutions of this equation are easily checked to be of the form $\sigma(k)=l+2k$, $k\in\IZ$, where $l=\sigma(0)$ is any integer.
In particular, the equality $\Ad(\rho_c(U^{-l}))(V\rho(S_2)V^*)=\rho_c(S_2)$ is got to at once. This concludes the proof, 
since  $\Ad(\rho_c(U^{-l}))(V\rho(U)V^*)= \Ad(\rho_c(U^{-l}))(\rho_c(U)) = \rho_c(U)$.
\end{proof}

Although $C^*(U)$ is a maximal abelian subalgebra of $\CQ_2$, \cite{ACR}, it is not true that the von Neumann algebra generated by $\rho(U)$ is maximal abelian in $\CB(\CH)$ for
any irreducible representation $\rho:\CQ_2\rightarrow\CB(\CH)$. Quite the opposite, permutative representations provide the easiest counterexamples we can think of.
Indeed,  thanks to the above result, if $\rho$ is a permutative representation inequivalent to the canonical representation, then the multiplicity of $\rho(U)$ will be bigger than one, which is the same as saying that 
$W^*(\rho(U))$ is not maximal.

\section{Decomposition of permutative representations of $\CQ_2$}\label{decomposition}

As is known,  any representation of a given  $C^*$-algebra decomposes into a direct sum of cyclic representations. In particular, this applies to permutative
representations of both $\CO_2$ and $\CQ_2$.  However, for these representations it is far more natural to seek decompositions into direct sums of cyclic
representations that are still of the same type. By \emph{permutative invariant subspace} of  a permutative representation  $\rho:\CQ_2\rightarrow\CB(\ell^2(\IN))$, therefore,  we shall
always mean a closed invariant subspace $M\subset\ell^2(\IN)$ given by $M=\overline{\textrm{span}}\{e_i: i\in I\subset\IN\}$, where $I$ is a suitable subset of $\IN$. It is obvious
that the restriction of $\rho$ to such an invariant subspace is still a permutative representation. In this case, we will also say that $\rho$ restricts to $M$ as a \emph{permutative subrepresentation},
borrowing the terminology from \cite{BJ}, where $\CO_2$ is dealt with. A permutative representation is  understood as \emph{cyclic} if it not only has a cyclic vector in the usual sense, but this can in fact be picked up  among 
those of the chosen orthonormal basis. It then turns out that every basis vector is cyclic. This is true of permutative representations of $\CQ_2$ as well as $\CO_2$. It takes a moment's
reflection to realize the proof is exactly the same in the two cases, which means we might safely rely on the results already proved in \cite{BJ}. Yet many of those basic results we need
are not really given a clear-cut statement there, so we have preferred to provide explicit proofs all the same.
To begin with, permutative representations enjoy a type of symmetry property that is worthy of a statement to itself. 
\begin{proposition}\label{symmetry}
Let $\pi:\CQ_2\rightarrow\CB(\ell^2(\IN))$ be a permutative representation. If the equality $\pi(S_\alpha S_\beta^*)e_k=e_h$ holds for $h,k\in\IN$, then the equality
$\pi(S_\beta S_\alpha^*)e_h=e_k$ holds as well. 
\end{proposition}
\begin{proof}
Since $\pi(S_\alpha)\pi(S_\beta^*)e_k=e_h\neq 0$, $\pi(S_\beta^*)e_k=\pi(S_\alpha^*)e_h$ are different from zero. In particular,
$\pi(P_\beta )e_k=S_\beta S_\beta^* e_k$ is not zero either, meaning $\pi(P_\beta)e_k=e_k$, in that $\pi(P_\beta)$ is a diagonal
projection with respect to the canonical basis $\{e_n:n\in\IN\}$. But then we have $e_k=\pi(P_\beta)e_k=\pi(S_\beta S_\beta^*)e_k=\pi(S_\beta)\pi(S_\alpha^*)e_h=\pi(S_\beta S_\alpha^*)e_h$, as claimed.
\end{proof}
Among other things, the above property allows for a straighforward proof that in a permutative cyclic representation all basis vectors are cyclic. 
\begin{proposition}
If $\pi$ is a cyclic permutative representation of $\CO_2$ or $\CQ_2$, then all basis vectors are cyclic.
\end{proposition}
\begin{proof}
We limit oursevels to treat $\CO_2$ only. Let $e_{i_0}$ be a cyclic vector. We first show that for every $k\in \IN$ there exists  at least one monomial
$S_\alpha S_\beta^*\in\CO_2$ such that $\pi(S_\alpha S_\beta^*)e_{i_0}=e_k$. Indeed, if this were not the case,
the linear subspace $\pi(\CO_2^{\textrm{alg}})e_{i_0}\subset\CH$, with $\CO_2^{\textrm{alg}}$ being the dense subalgebra generated by monomials $S_\mu S_\nu^*$ for $\mu,\nu\in W^2$, would fail to contain $e_k$. More precisely, every vector
$x\in\pi(\CO_2^{\textrm{alg}})$ would satisfy $(x, e_k)=0$. By density, we would finally find that every $x\in\overline{\pi(\CO_2)e_{i_0}}$ is such that $(x, e_k)=0$, which contradicts the cyclicity of $e_{i_0}$. 
It is now clear how to get to the conclusion. Indeed, thanks to Proposition \ref{symmetry} we know
$\pi(S_\beta S_\alpha^*)e_k=e_{i_0}$ , which immediately makes it plain $e_k$ is also cyclic.
\end{proof}
The above result might beguile the reader into thinking that cyclic permutative representations are automatically irreducible. The situation is in fact  a shade more involved than that, and cyclic permutative representations
will in general fail to be irreducible. Even so, the natural condition that prevents this from happening  has been spotted by Bratteli and Jorgensen in \cite{BJ}, where they introduce a suitable notion of multiplicity-free
permutative representation of $\CO_2$, which we next recall in some detail. To this aim, we first need to  recall the definition of the so-called \emph{coding map} as it is introduced in \cite{BJ}.
If $(\sigma_1, \sigma_2)$ is a branching function system of order $2$, we can define a map $\sigma:\IN\rightarrow\{1,2\}^\IN$ in the following way.
For any given $n\in\IN$, there is only one sequence $(i_1,i_2,\ldots, i_k,\ldots)\in\{1,2\}^\IN$ such that $n$
lies in the range of $\sigma_{i_1}\circ\sigma_{i_2}\circ\ldots\circ\sigma_{i_k}$  for any $k\in\IN$: by definition, the value $\sigma(n)$ of the coding map on the integer $n$ is just this sequence. 
In terms of the Cuntz isometries, if $\pi$ is the representation associated with $(\sigma_1,\sigma_2)$, the sequence $\sigma(n)=(i_1,i_2,\ldots, i_k,\ldots)$ is completely determined by the condition
$\pi(S_{i_k}^*\ldots S_{i_2}^* S_{i_1}^*)e_n\neq 0$ for every $k\in\IN$. Now a permutative representation  is \emph{multiplicity-free} in the sense of Bratteli-Jorgensen if the corresponding coding map is injective.  
More generally, our representation is said \emph{regular} if its coding map is only \emph{partially injective}, namely if $\sigma(n) = \sigma(\sigma_{i_i} \circ \ldots \circ \sigma_{i_k}(n))$ then
$n = \sigma_{i_i} \circ \ldots \circ \sigma_{i_k}(n)$, for any $i_1, i_2, \ldots i_k\in\{1,2\}$ and $k\in\IN$. Obviously, both definitions continue to make sense for a representation $\rho$ of $\CQ_2$ also, since it is quite natural to say $\rho$ is multiplicity-free or regular
if its restriction to the Cuntz algebra $\CO_2$ is. At this point, it is already fairly clear that a permutative representation that decomposes into the direct sum of multiplicity-free permutative subrepresentations is regular.
That said, we can move on to discuss the announced result. First, we single out in the following lemma a  separation property enjoyed by multiplicity-free representations. 

\begin{lemma}\label{separation}
Let $\pi:\CO_2\rightarrow\CB(\CH)$ be a multiplicity-free permutative representation. Then forn any finite set of basis vector $\{e_{n_0}, e_{n_1,},\ldots, e_{n_k} \}$ there
is a multi-index $\alpha\in W_2$ such that $\pi(S_\alpha S_\alpha^*)e_{n_0}=e_{n_0}$ and $\pi(S_\alpha S_\alpha^*)e_{n_i}=0$ for every $i=1,2, \ldots, k$.
\end{lemma}
\begin{proof}
Let $\alpha_l=(i_1,i_2,\ldots, i_l)$ be the multi-index of length $l$ obtained out of $\sigma(n_0)=(i_1,i_2,\ldots)$ by taking its first $l$ values.
We have $\pi(S_{\alpha_l}S_{\alpha_l}^*)e_{n_0}=e_{n_0}$ for every $l\in\IN$ by construction. We argue by induction on $k$. If $k=1$, there must exist 
an $l\in\IN$ such that $\pi(S_{\alpha_l}S_{\alpha_l}^*)e_{n_1}=0$, for otherwise $\sigma(n_1)$ would be the same as $\sigma(n_0)$, which is not possible by hypothesis. Finally, the inductive
step can be taken in much the same way. 
\end{proof}

\begin{proposition}\label{cycirr}
Any multiplicity-free cyclic permutative representation $\pi$ of $\CO_2$ or $\CQ_2$ is irreducible.
\end{proposition}

\begin{proof}
Again, we  limit ourselves to dealing with $\CO_2$ only. We need to prove that every non-zero vector is cyclic. To this end, it is enough to produce a dense subset of cyclic vectors, and  this is given by finite linear combinations of the form
$x=\lambda_1e_1+\lambda_2e_2+\ldots\lambda_ne_n$, $n\in\IN$. If $x$ is not zero, we may suppose $\lambda_1\neq 0$ without loss of generality. 
By applying Lemma \ref{separation} we see there exists a certain $\alpha\in W_2$ such that $\pi(S_\alpha S_\alpha^*)e_1=e_1$  and
$\pi(S_\alpha S_\alpha^*) e_j=0$ for every $j=2,3,\ldots, n$. But then $\pi(\frac{1}{\lambda_1}S_\alpha S_\alpha^*)x=e_1$, and so
$\overline{\pi(\CO_2)x}$ exhausts $\CH$, since it contains the cyclic vector $e_1$.
\end{proof}
Now by means of a standard application of Zorn's lemma one sees at once that any permutative representation decomposes into the direct sum of
cyclic permutative representations. 
\begin{proposition}
Every permutative representation of $\CO_2$ or $\CQ_2$ decomposes into the direct sum of cyclic permutative representations. 
\end{proposition}
\begin{proof}
For instance, let $\pi$ be a permutative representazion of the Cuntz algebra. We only need to show that $\pi$ has a cyclic permutative subrepresentation. Now a subrepresentation of this type can be produced at once by considering the closed
subspace $M\doteq\overline{ \pi(\CO_2)e_0}$. Indeed, $M$ is cyclic and $\pi$ invariant by construction. The conclusion is then reached  by realizing that there exists an orthonormal basis of $M$ made up of basis vectors, that is $M=\overline{\textrm{span}}\{e_i: i\in I\subset\IN\}$, and this is easily seen by noting that $M$ can also be described as $\overline{\textrm{span}}\{\pi(S_\alpha S_\beta^*)e_0: \alpha,\beta\in W_2\}$.
\end{proof}   
Furthermore, if the representation $\pi$ is also multiplicity-free, in the above decomposition all cyclic representations are actually irreducible thanks to Proposition \ref{cycirr}. However, in  \cite[Theorem 2.7]{BJ}
more is proved. Indeed, any
multiplicity-free permutative representation is shown to decompose into the direct sum of irreducible permutative subrepresentions that are in addition pairwise inequivalent. Notably, the representation is multiplicity-free in the usual sense as well, i.e.
the commutant $\pi(\CO_2)'$ is abelian.
If the representation is only assumed to be regular, then it is still completely reducible, but permutative irreducible subrepresentations may appear with multiplicity greater than one. 
Curiously enough, it is nowhere explicitly said in their monograph that the converse, too, holds true, although the authors were perhaps fully aware of this fact.
Be that as it may, this further confirms that the two definitions could not possibly have been any better. At any rate, it is not quite a matter of being nuanced about how good the definitions are; it is more that
we do need the converse, see Therem \ref{BJconverse}, to prove our own results. Because the proof is not entirely obvious, it is postponed to the next section. 
 
\section{A converse to a theorem of Bratteli-Jorgensen}\label{converse}
Cyclic permutative representations of $\CO_2$ are completely known.  More precisely, they can all be classified in terms of a (possibly infinite) multi-index $I$ in the alphabet $\{1,2\}$.
Associated with any such index $I$ there is a unique cyclic permutative representation $\pi_I$ of $\CO_2$  that is completely determined by the following properties. 
When the multi-index $I$ is finite, say $I=(i_1,i_2,\ldots, i_k)$, then there is a unique basis vector $\Omega$ such that $\pi_I(S_I)\Omega=\Omega$ and the set of vectors
$\{\Omega, \pi_I(S_{I_1})\Omega,\pi_I(S_{I_2})\Omega,\ldots, \pi_I(S_{I_{k-1}})\Omega \}$ is an orthonormal system, where $I_j$ is the multi-index obtained out of $I$ by considering its first $j$ entries only. 
When the multi-index is an infinite sequence $I=(i_1, i_2,\ldots, i_n,\ldots)$ instead, $\Omega$ is the only basis vector such that
$S_{i_n}^*\ldots S_{i_2}^*S_{i_1}^*\Omega\neq 0$ for every $n\in\IN$ and the set $\{\Omega, \pi_I(I_k)\Omega: k\in\IN\}$ is an orthonormal system. In \cite{Kaw06} $\pi_I$
is referred to as the representation of type $P(I)$, and so is it in the present work.  More importantly, a representation of type $P(I)$ is irreducible if and only if either
the multi-index $I$ is finite and cannot be written as $JJ\ldots J$, where $J$ is another finite multi-index such that $|J|$ divides $|I|$, or it is infinite and is not eventually periodic, see e.g.
\cite[Theorem 3.4]{DavPitts}. Interestingly, a representation rising from  a multi-index whose length is an odd prime number is necessarily irreducible. Finally, note that
$P(1)$ and $P(2)$ are  respectively nothing but $\pi_-$ and $\pi_+$.
That said, we can move on to the announced result. In order to prove it, the first step to take is to ensure that irreducible representations are multiplicity-free
in the sense of Bratteli-Jorgensen. To do that, we first need to give a definition. We say that a monomial $S_\mu S_\nu^*$ is \emph{reduced} if it cannot be written as $S_{\mu'} P_\beta S_{\nu'}^*$ with $\mu=\mu' \beta$, 
$\nu = \nu'\beta$ and $P_\beta=S_\beta S_\beta^*$ a non-trivial standard diagonal projection. Note that if in a permutative representation $S_{\alpha} P_\beta S_{\gamma}^* e_n \neq 0$ then $S_{\alpha} P_\beta S_{\gamma}^* e_n = S_{\alpha} S_{\gamma}^* e_n$.
A couple of observations are now in order to better understand what the definition actually rules out. First, if either $\mu$ or $\nu$ is empty, the corresponding monomial is certainly reduced; in particular, the identity $I$ is reduced. 
Second, none of the non-trivial standard diagonal projections are reduced. 
\begin{thm} \label{irrmf}
Any irreducible permutative representation of $\CO_2$ is multiplicity-free in the sense of Bratteli-Jorgensen. 
\end{thm}
\begin{proof}
There are two cases to deal with according to whether for any given $k\in\IN$ the equality $S_\mu S_\nu^*e_k=S_{\mu'}S_{\nu'}^*e_k \neq 0$ 
with reduced $S_\mu S_\nu^*$, $S_{\mu'}S_{\nu'}^*$
implies $\mu=\mu'$ and $\nu=\nu'$.
% This excludes the case called P(I) by Kawamura, with $I= i_1 \ldots i_k$ with $\Omega$ a basis vector.
If this is the case, the proof to the theorem is reached by contradiction. Assuming that the coding map fails to be injective, say $\sigma(n)=\sigma(m)$ with $n\neq m$,
a bounded linear operator $T$ is well defined on the Hilbert space $\ell^2(\N)$ by $TS_\mu S_\nu^*e_n\doteq S_\mu S_\nu^*e_m$, for any pair of (reduced) multi-indices $\mu, \nu$. 
Indeed, if $S_\mu S_\nu^*e_n$ is zero then $S_\mu S_\nu^*e_m$ is also zero, since $\sigma(n)=\sigma(m)$ by assumption. By irreducibility the linear spans of
$\{S_\mu S_\nu^*e_n: \mu,\nu\}$ and $\{S_\mu S_\nu^*e_m: \mu,\nu\}$ are both dense, which means $T$ is densely defined with dense range. Because it is clearly isometric as well, it 
extends to a unitary operator of $\ell^2(\IN)$. Furthermore, sending $e_n$  to $e_m$, $T$ is  not a multiple of the identity.  
%The operator $T$ is then easily seen to be unitary. 
The contradiction is finally arrived at if we show that $T$ lies in the commutant of the representation. This is quickly verified. As $T$ is unitary, we only need to make sure that
$S_iT=TS_i$, $i=1,2$. These equalities are immediately seen to hold true at the level of $\textrm{span}\{S_\mu S_\nu^*e_n: \mu,\nu\}$, for
$S_iT S_\mu S_\nu^*e_n= S_iS_\mu S_\nu^*e_m=TS_iS_\mu S_\nu^*e_n$, $i=1,2$.\\
In the second case, it is not difficult to realize that the representation is of type $P(I)$ in the sense of Kawamura, see \cite{Kaw06}, so that there exists 
$\Omega\in\ell^2(\IN)$ such that $S_I\Omega=\Omega$ for some finite multi-index $I=(i_1,i_2,\ldots, i_k)\in W_2$. 
% and ... is an orthogonal system
Moreover, the vector $\Omega$ is a suitable basis vector, say $\Omega=e_h$. In particular, the value of the coding map at $h$ is given by
$\sigma(h)=\overline{i_1i_2\ldots i_k}=\overline{I}$, where we have adopted the convention 
that $\overline{\alpha}$ is the periodic sequence in which $(\alpha_1,\alpha_2,\ldots \alpha_{|\alpha|})$ is repeated infinitely many times.
%Using a result by Kawamura, the coding map is then proved to be injective. 
We now want to prove that the coding map $\sigma$ is injective. To this aim, it is useful to observe that $\CH$ can be obtained as the closed span of the set
$\{S_\alpha\Omega: \alpha\in W_2\}$. In other words, for any $n\in\IN$ there exists at least one multi-index $\alpha\in W_2$ such that
$S_\alpha\Omega= e_n$. As a result, we see that $\sigma(n)$ is nothing but the sequence $\alpha\bar{I}$. From this, the coding map is easily seen to be injective.
Indeed, if for $n,m\in\IN$ we have $\sigma(n)=\sigma(m)$, then $\alpha\bar{I}=\beta\bar{I}$, where $\alpha,\beta\in W_2$ are two multi-indeces
such that $e_n=S_\alpha\Omega$ and $e_m=S_\beta\Omega$. Apart from the trivial case when $\alpha=\beta$, the equality $\alpha\bar{I}=\beta\bar{I}$ is
still possible if $\alpha$ and $\beta$ differs by a multiple of $I$, i.e. $\beta=\alpha (kI)$, for some $k\in\IN$, since $I$ is an irreducible block.
But in this case $e_m=S_\beta\Omega= S_\alpha S_{kI}\Omega=S_\alpha\Omega= e_n$, hence $n=m$.
\end{proof}
As an easy consequence, we now have the following result too. 
\begin{corollary}\label{regular}
A permutative representation $\pi:\CO_2\rightarrow\CB(\CH)$ that decomposes into the direct sum of irreducible permutative subrepresentations is regular in the
sense Bratteli-Jorgensen.
\end{corollary}
\begin{proof}
We need to check that the coding map $\sigma$ associated with $\pi$ is partially injective, namely 
if $n\in\IN$ is such that $\sigma(n)=\sigma(\sigma_{i_1}\circ \sigma_{i_2} \circ \ldots \circ \sigma_{i_k}(n))$ then $n=\sigma_{i_1}\circ \sigma_{i_2} \circ \ldots \circ \sigma_{i_k}(n)$.
Under our hypotheses, $\IN$ may be written as (at most countable)  disjoint union of infinite set $A_i$ such that $\CH_i\doteq \overline{\textrm{span}}\{e_n: n\in A_j\}\subset\CH$
is an irreducible $\pi$-invariant subspace.  Now, given any $n\in\IN$ there is a unique $i_0$ such that $n\in A_{i_0}$, that is $e_n\in\CH_{i_0}$.
With a slight abuse of notation we write $\sigma(e_n)$ instead of $\sigma(n)$. If we do so, we see that $\sigma(\sigma_{i_1}\circ \sigma_{i_2} \circ \ldots \circ \sigma_{i_k}(n))$ is nothing
but $\sigma(S_{i_1}S_{i_2}\ldots S_{i_k}e_n)$, and the conclusion is immediately got to since $S_{i_1}S_{i_2}\ldots S_{i_k}e_n$ is still in $\CH_{i_0}$ and the restriction
of $\sigma$ to $A_{i_0}$ is injective.
\end{proof}

Furthermore, if each irreducible component shows up at most once, i.e. our representation is multiplicity-free in the general sense, then the coding
map can be shown to be injective. This finding is nothing but the sought converse to the theorem of Bratteli and Jorgensen that a multiplicity-free
permutative representation decomposes into the direct sum of inequivalent irreducible permutative subrepresentations, cf. \cite[Theorem 2.7]{BJ}.

\begin{thm}\label{BJconverse}
Let $\pi$ be a permutative representation of $\CO_2$ such that  $\pi$ decomposes as a direct sum of inequivalent irreducible permutative subrepresentations.
Then the coding map is injective, that is $\pi$ is multiplicity-free in the sense of Bratteli-Jorgensen.
\end{thm}

\begin{proof} 
By our assumptions, $\IN$ decomposes into the disjoint union of (at most countably infinitely many) subsets $A_i$, $i\in I$, such that
$H_i\doteq\overline{\textrm{span}}\{e_j: j\in A_i\}\subset\ell^2(\IN)$ are inequivalent irreducible subspaces of the given representation $\pi$.
Given two different integer numbers $n,m\in\IN$, there are two cases to deal with according to whether $n$ and $m$ both belong to an $A_i$, for some $i\in I$.
If they do, there is not much to prove in view of the result above, since the restriction of $\sigma$ to each of the $A_i$'s has already been shown to be injective.
Suppose this is not the case. Then $n\in A_i$ and $m\in A_j$, with $i\neq j$. We now claim that $\sigma(n)=\sigma(m)$ implies $\omega_n=\omega_m$, where
$\omega_k$ is the vector state associated with the basis vector $e_k$, i.e. $\omega_k(x)\doteq (\pi(x)e_k, e_k)$, $x\in\CO_2$.
 But this leads to an absurd, because $\pi\upharpoonright_{H_n}$ and
$\pi\upharpoonright_{H_m}$ are then unitarily equivalent, since they are two GNS representations arising from the same state.  
In order to prove the claim, we actually show a bit more. In fact, the sequence $\sigma(k)$ contains full information about the state
$\omega_k$. More precisely, knowing exactly what $\sigma(k)$ is allows us to evaluate $\omega_k$ at every monomial
$S_\mu S_\nu^*$. Indeed, by definition $\omega_k(S_\mu S_\nu^*)=(S_\mu S_\nu^*e_k, e_k)=(S_\nu^*e_k, S_\mu^*e_k)$. There are now a few cases 
to treat. If one of $\mu$ or $\nu$ is not an initial word of $\sigma(k)$, then $\omega_k(S_\mu S_\nu^*)$ is zero. If
both $\mu$ and $\nu$ are initial words of $\sigma(k)$, then either $|\mu|=|\nu|$ or $|\mu|\neq |\nu|$. In the first case, we necessarily
have $\mu=\nu$, which means $\omega_k(S_\mu S_\nu^*)=1$. In the second case, we can suppose $|\mu| >|\nu|$, which means
$\mu=\nu\tilde{\mu}$, and so $\omega_k(S_\mu S_\nu^*)$ is given by $(S_\nu^*e_k, S_{\tilde{\mu}}^* S_\nu^*e_k)$. 
To conclude, we need to note that $e_l\doteq S_\nu^*e_k$ and $e_{l'}\doteq S_{\tilde{\mu}}^*S_\nu^* e_k$ lie in the same irreducible subspace, namely $H_{i_0}$ with $k\in A_{i_0}$. 
Now $\omega_k(S_\mu S_\nu^*)=1$ if and only if $l=l'$ and $0$ otherwise. But the injectivity of $\sigma$ restricted to $A_{i_0}$ says that
$\omega_k(S_\mu S_\nu^*)$ is $1$ if $\sigma(l)=\sigma(l')$ and $0$ if $\sigma(l)\neq \sigma(l')$. But both $\sigma(l)$ and $\sigma(l')$ are infinite subsequences of
$\sigma(k)$ by construction, obtained by removing the first $|\nu|$ and $|\mu|$ digits, respectively, which ends the proof.
\end{proof}

In particular, the restriction of the canonical representation of $\CQ_2$ to $\CO_2$ is multiplicity-free in the sense of Bratteli-Jorgensen, since it decomposes into the direct sum
of two inequivalent representations $\pi_+$ and $\pi_-$, cf. \cite[Sect. 2.2]{ACR}. While $\pi_+$ and $\pi_-$ are inequivalent, they can still be obtained out of each other by composing with 
the flip-flop automorphism $\lambda_f$, which is given by $\lambda_f(S_1)=S_2$ and $\lambda_f(S_2)=S_1$. Indeed, we have the following result.
\begin{proposition}
If $V$ is the unitary between $\CH_-$ and $\CH_+$ given by $Ve_k\doteq e_{-k-1}$, then $V\pi_-=(\pi_+\circ\lambda_f)V$.
\end{proposition}
\begin{proof}
All we need to do is to make sure that the equalities $V\pi_-(S_1)=\pi_+(S_2)V$ and $V\pi_-(S_2)=\pi_+(S_1)V$ hold. But these are easily verified by a direct computation.
\end{proof}
It is also worth pointing out that the operator $V$ is  the unique permutative unitary that intertwines $\pi_-$ and $\pi_+\circ\lambda_f$.\\

At this point, it is useful to point out that representations of type $P(I)$ may well fail to be regular when they are not irreducible, i.e. when the corresponding multi-index $I$ is made up of a finite number
of finite blocks or it is infinite but eventually periodic. In this situation, the theorem of Bratteli and Jorgensen does not apply as the coding map is no longer partially injective.  Even so, a decomposition into irreducible
components still exists, although some of these will not be permutative subrepresentations. It is then necessary to consider more general representations, where a phase can appear as well as a permutation
of the basis vectors.  These are the so-called representations of type $P(J,z^k)$, with $|J|=k$ and $z \in \IT$, cf. \cite{Kaw06} and
the references therein, which can be simply described as the composition P$(J) \circ \alpha_{z}$, where
$\{\alpha_z:z\in\IT\}$ are the gauge automorphisms of $\CO_2$, namely $\alpha_z(S_i)=zS_i$, $i=1,2$.
Obviously, a representation of type   $P(J,z^k)$ is irreducible if and only if $P(J)$ is. Moreover, it extends to a representation
of $\CQ_2$ if and only if $P(J)$ does, for $\alpha_z$ certainly extends, see \cite{ACR}.
For instance, it can be seen that $P(1212)$ decomposes as $P(12) \oplus P(12,-1)$. 
For what follows, we also need to remark that the representations  $P(1_k)$ and $P(2_k)$ are completely reducible for every $k\in\IN$, where $1_k$ and $2_k$ are the multi-indices of
length $k$ whose entries are all $1$'s or $2$'s, respectively. Indeed, it is proved in \cite{Kaw06} that 
$P(1_k)=P(1)\oplus P(1,\zeta)\oplus\ldots\oplus P(1, \zeta^{k-1})=P(1)\oplus (P(1)\circ\alpha_\zeta)\oplus\ldots\oplus (P(1)\circ\alpha_{\zeta^{k-1}})$ and
$P(2_k)=P(2)\oplus P(2,\zeta)\oplus\ldots\oplus P(2, \zeta^{k-1})=P(2)\oplus (P(2)\circ\alpha_\zeta)\oplus\ldots\oplus (P(2)\circ\alpha_{\zeta^{k-1}})$, with
$\zeta=e^{\frac{2\pi i}{k}}$. Now neither $P(1_k)$ nor $P(2_k)$ extends to $\CQ_2$ because in $P(1_k)$  the point spectrum of $S_1$ is the set of all $k$-th roots of unity, with each eigenvalue being
simple, whereas the point spectrum of $S_2$ is empty, and the other way round in $P(2_k)$.
However, their direct sum $P(1_k)\oplus P(2_k)$ does extend by virtue of the equality
$$P(1_k)\oplus P(2_k)=\pi_c\oplus (\pi_c\circ\alpha_\zeta)\oplus\ldots\oplus(\pi_c\circ\alpha_{\zeta^{k-1}})$$
which holds up to equivalence. This also shows that there is at least one non-permutative extension of $P(1_k)\oplus P(2_k)$ to $\CQ_2$ that is not irreducible, which is simply  given by 
$$\rho_c\oplus (\rho_c\circ\widetilde\alpha_\zeta)\oplus\ldots\oplus(\rho_c\circ\widetilde \alpha_{\zeta^{k-1}})$$
Yet this does not still say that the permutative extensions of $P(1_k)\oplus P(2_k)$ are reducible themselves. 
In fact, this is just the case. The proof, though, requires a more painstaking analysis, which is carried out in the next section.

\section{Irreducible representations of $\CQ_2$} \label{irreducible}

We can now return to the discussion of permutative representations of $\CQ_2$. In particular, we would like to focus our attention on  irreducible representations.
To this aim, we need to improve our knowledge of the irreducible representations of type
$P(I)$ at the level of the Cuntz algebra. More precisely, we first need to answer the question of whether they extend to $\CQ_2$ or not.   It is somewhat surprising that they all do apart from
$P(1)$ and $P(2)$.  Not only do they extend, but their extension is also unique.

\begin{proposition}\label{irrext}
In all  irreducible representations of type $P(I)$ with $| I |\geq 2$ the generating isometries are pure. 
In particular, apart from $P(1)$ and $P(2)$,  all of these irreducible representations uniquely extend to $\CQ_2$. 
\end{proposition}

\begin{proof}
We first handle the case of a finite multi-index $J=(j_1, j_2,\ldots, j_k)$. Following Kawamura, cf. \cite[Lemma 2.2]{Kaw06}, the isometries $S_1$ and $S_2$ can be realized concretely
by the branching function system $(\sigma_1, \sigma_2)$ given by
$$\sigma_1(1) = \begin{cases} k+1 & j_1=2 \\ k & j_1 = 1 \end{cases}, \quad
\sigma_1(l) = \begin{cases} k+l & j_l=2 \\ l-1 & j_l = 1 \end{cases}, 2 \leq l \leq k, \quad f_1(l) = 2l-1, l \geq k+1$$
$$\sigma_2(1) = \begin{cases} k & j_1=2 \\ k+1 & j_1 = 1 \end{cases}, \quad
\sigma_2(l) = \begin{cases} l-1 & j_l=2 \\ k+l & j_l = 1 \end{cases}, 2 \leq l \leq k, \quad f_2(l) = 2l, l \geq k+1$$
To begin with, it is worth noting that $S_J e_k=e_k$, which follows from the equalities $\sigma_{j_k}(k)=k-1$ and $\sigma_{j_1}(1)=k$. 
Proving that $S_1$ and $S_2$ are pure amounts to making sure that  for any $l\in\IN$ $\sigma_1^n(l)$ and $\sigma_2^n(l)$ diverge as $n\rightarrow\infty$.
We only need to worry about the first sequence, as the second can be worked out with in much the same way.  If for some $l_0\in\IN$
the sequence $\{\sigma_1^n(l_0): n\in\IN\}$ were bounded, then there would exist an integer $n_0$ such that $\sigma_1^{n_0}(l_0)=l_0$, whence
$S_1^{n_0}e_{l_0}=e_{l_0}$, meaning our representation would be a representation of type $P(11\ldots1)$, which obviously is not the case, since
the representation $P(11\ldots 1)$ is irreducible only when $k=1$. \\
The case of an infinite multi-index is still easier to deal with. Again, following Kawamura, cf. \cite[Lemma 2.3]{Kaw06}, the isometries
$S_1$ and $S_2$ are now implemented by the two injective functions $f_1, f_2: \IZ\times\IN\rightarrow \IZ\times\IN$ given by
$$f_1(n,1)=(n-1, p_n(1))\quad f_1(n,m)=(n-1, 2m-1)\,\,\textrm{for}\,\,m\geq 2$$
$$f_2(n,1)=(n-1, p_n(2))\quad f_2(n,m)=(n-1,2m)\,\,\textrm{for}\,\, m\geq 2$$
where, for every $n\in\IZ$,  $p_n$ is  the permutation on $\{1, 2\}$ such that $p_n=\id$ if $n\leq 0$ and $p_n(1)=j_n$ otherwise. 
The pureness is then proved as soon as one realizes that neither $f_1$ nor $f_2$ has finite orbits.   
\end{proof}

From now on we will denote the unique extension of $P(I)$ to $\CQ_2$ by $\widetilde P(I)$. 
Although $\widetilde P(I)$ is completely determined by its restriction to the Cuntz algebra, in general it is not an entirely trivial task to see explicitly how $U$ acts in this representation.  
In addition, saying what its multiplicity is for any given multi-index $I$ is far from obvious, not least because the formulas one obtains may be rather unwieldy to compute with. 
One can already get a better grasp of the problem  by  analysing the following elementary example, where as simple a case as $\widetilde P(12)$ is looked at more closely. 

\begin{example}
We discuss two different realizations of $P(12)$ at the level of $\CO_2$, which of course yield two different yet unitarily equivalent realizations
of $\widetilde P(12)$.  As usual, we consider the Hilbert space $\ell^2(\IN)$ endowed with the canonical basis $\{e_n: n\in\IN\}$. 
In both representations $S_2$ simply acts as $S_2e_k\doteq e_{2k}$, $k\in\IN$. In the first representation, $S_1$ acts as the usual isometry
%$e_k\rightarrow e_{2k-1}$ 
composed with a switch on the first two basis vectors, i.e. $S_1e_1\doteq e_3$, $S_1e_2\doteq e_1$, and
$S_1e_k\doteq e_{2k-1}$, for every $k\geq 3$. In the second, the switch is performed pairwise on all vectors instead. More explicitly, $S_1$ now takes the form
$S_1e_k\doteq e_{2k+1}$ if $k$ is odd, and $S_1 e_k\doteq e_{2k-3}$ if $k$ is even.\\
%\begin{proposition}
%The above representations of $\CO_2$ on $\ell^2$ are irreducible.
%\end{proposition}
%\begin{proof}
%In both representations the isometry $S_1S_2$ has $\IC e_1$ as its unique eigenspace, as easily checked.  If now $E$ is any orthogonal projection in the commutant
%of either representation, we have $Ee_1=ES_1S_2e_1=S_1S_2Ee_1$, whence $Ee_1$ is either $0$ or $e_1$ thanks to the remark we made above.
%The case $Ee_1=0$ leads to $E=0$, whereas $Ee_1=e_1$ leads to $E=1$.
%\end{proof}
%A straightforward application of Proposition \ref{cor-pureness} proves that the generating isometries of $\CO_2$ are pure in both these representations, which 
%% Since the images of $S_1$ and $S_2$ in these representations are always pure, 
%therefore uniquely extend to (automatically permutative) representations of $\CQ_2$.
%Furthermore, in both cases the image of the unitary $U$ can be described explicitly.
In the first, one easily sees that $Ue_2=e_3$, $Ue_4 = e_1$ and $U e_{2k} = e_{2k-1}$ for all $k \geq 3$. For odd numbers one finds out $Ue_1 = e_6$, $U e_{11} = e_{10}$,
and $U e_{2k-1} = e_{2k-2}$ for all $k \geq 8$, 
 with the exception of the two sequences 
 % with initial conditions $U e_{17} = e_{96}$ and $U e_{25} = e_8$
 $U e_{k_0} = e_{h_0}$, $Ue_{2^n k_0 - \sum_{i=0}^{n-1} 2^i} = e_{2^n h_0}$ for all $n=1,2,3,\ldots$
with $(k_0,h_0) $ equal to $(3,12)$ and $(7,2)$, respectively.\\
In the second, after making some computations one finds out that the action of $U$ can be described rather easily on even numbers, and yet much less so on odd numbers. 
At any rate, the formulas arrived at are the following.
\begin{align*}
U e_{2n} & = 
\begin{cases} 
e_{2n+1} & n \ \textrm{odd} \\
e_{2n-3} & n \ \textrm{even}
\end{cases} \\
Ue_{2^k n - (2^{k-1}+1)} & =  
\begin{cases} 
e_{2^k n + (2^k - 2^{k-1})} & n \ \textrm{odd} \\
e_{2^k n - (2^k + 2^{k-1})} & n \ \textrm{even}
\end{cases} 
\end{align*}
where $k$ is any integer bigger than or equal  to $2$. The above formula actually defines $U$ on all odd vectors too, since every odd number can always be written 
as $2^kn-(2^{k-1}-1)$ for  suitable $k\geq 2$ and $n\in\IN$.
\end{example}

There would be no a priori reason to expect the representations $\widetilde P(I)$ to exhaust all irreducible representations of $\CQ_2$.  In principle, there might
be many irreducible permutative representations of $\CQ_2$ that restrict to $\CO_2$ as reducible representations. The canonical representation $\rho_c$ is just such an example. 
However, that is the one example, for it turns out that any permutative irreducible representation $\rho$ of $\CQ_2$ restricts to $\CO_2$ as an irreducible representation as long as $\rho$ is not
unitarily equivalent to $\rho_c$. In order to prove this, however, we are yet to fully analyze the permutative extensions of representation of $\CO_2$ of the form $P(1_k)\oplus P(2_k)$, which is done in the next couple of propositions.
In particular, what we aim to do is show every such representation is not irreducible in that it contains a copy of the canonical representation.

\begin{proposition}
Let $\rho:\CQ_2\rightarrow\CB(\CH)$ be a permutative representation such that there exists a permutative $\CO_2$-invariant subspace
$\CH(1_k)\subset \CH$ on which $\rho\upharpoonright_{\CO_2}$ acts as a representation
of type $P(1_k)$. Then there also exists a permutative $\CO_2$-invariant subspace $\CH(2_k)\subset \CH$ on which 
$\rho\upharpoonright_{\CO_2}$ acts as a representation of type $P(2_k)$. Moreover, $\CH(1_k)$ and $\CH(2_k)$ are orthogonal
and their direct sum $\CK\doteq\CH(1_k)\oplus\CH(2_k)$ is a $\CQ_2$-invariant subspace, i.e. $\CK$ is a permutative subspace that reduces $\rho$.
\end{proposition}
\begin{proof}
By hypothesis there exists a basis vector $\Omega_1$ such that $\rho(S_1)^k\Omega_1=\Omega_1$ and
$\{\Omega_1, \rho(S_1)\Omega_1,\ldots,\rho(S_1)^{k-1}\Omega_1\}$ is an orthonormal set. 
It is then easy to see that $\CH(1_k)$ is given by $\overline{\textrm{span}}\{S_\mu\Omega_1: \mu\in W_2\}$. 
In other terms, in this situation there is actually no need to consider more general monomials of
the form $S_\mu S_\nu^*$, since $S_2^*\Omega_1=0$ and $(S_1^*)^l\Omega_1=S_1^{k-l}\Omega_1$ for every
$l=1,2,\ldots,k$. We now define $\Omega_2\doteq \rho(U)\Omega_1$. Obviously, $\Omega_2$ is still a basis vector.
Furthermore, we also have $\rho(S_2^k)\Omega_2=\Omega_2$. Indeed,  
$$
\rho(S_2)^k\Omega_2=\rho(S_2^kU)\Omega_1=\rho(US_1^k)\Omega_1=\rho(U)\rho(S_1^k)\Omega_1=\rho(U)\Omega_1=
\Omega_2
$$
Since the set $\{\Omega_2,\rho(S_2)\Omega_2,\ldots, \rho(S_2^{k-1})\Omega_2\}$ is clearly orthonormal, the $\CO_2$-invariant subspace $\CH(2_k)\doteq \overline{\textrm{span}}\{S_\mu\Omega_2:\mu\in W_2\}$ yields a subrepresentation of type
$P(2_k)$. Simple but tedious computations show that $\CH(1_k)\perp\CH(2_k)$, i.e. $(\rho(S_\mu)\Omega_1, \rho(S_\nu)\Omega_2)=0$ for every $\mu,\nu\in W_2$. All is left to do is check that $\CK\doteq \CH(1_k)\oplus \CH(2_k)$ is also invariant for both $\rho(U)$ and $\rho(U^*)$. After a moment's reflection, one easily realizes that verifying this invariance property amounts to showing that  $\rho(U)\Omega_2$ and $\rho(U^*)\Omega_1$ are still in $\CK$. But this is certainly the case thanks to the equalities $\rho(U)\Omega_2=\rho(S_1)\rho(S_2^{k-1})\Omega_2$ and $\rho(U^*)\Omega_1=\rho(S_2)\rho(S_1^{k-1})\Omega_1$, which can be proved starting by
$\rho(S_2^k)\Omega_2=\Omega_2$ and $\rho(S_1^k)\Omega_1=\Omega_1$ respectively.
\end{proof}
 
\begin{proposition}\label{caninside}
Let $\rho:\CQ_2\rightarrow\CB(\CH)$ be a permutative representation such that there exists an $\CO_2$-invariant subspace
$\CH(1_k)\subset \CH$ on which $\rho\upharpoonright_{\CO_2}$ acts as a representation
of type $P(1_k)$. Then $\rho$ contains a copy of the canonical representation. In particular, $\rho$ is not irreducible. 
\end{proposition}
\begin{proof}
Let $\CK=\CH(1_k)\oplus\CH(2_k)\subset \CH$ the $\rho$-invariant subspace we produced above. If we define
$$
\Phi_1\doteq\frac{1}{\sqrt{k}}(\Omega_1+\rho(S_1)\Omega_1+\ldots+\rho(S_1^{k-1})\Omega_1)\quad \textrm{and}\quad
\Phi_2\doteq\frac{1}{\sqrt{k}}(\Omega_2+\rho(S_2)\Omega_2+\ldots+\rho(S_2^{k-1})\Omega_2)
$$
we immediately see that $\rho(S_i)\Phi_i=\Phi_i$, $i=1,2$, and $U\Phi_1=\Phi_2$. Note that neither $\Phi_1$ nor $\Phi_2$ is a
basis vector. However, if we set $\CK(i)=\overline{\textrm{span}}\{S_\mu\Phi_i:\mu\in W_2\}$, $i=1,2$, the same proof as above shows that $K(1)\oplus K(2)$ is a (proper) $\rho$-invariant subspace, on which $\rho$ restricts as the canonical representation. 
\end{proof}

\begin{remark}
The subrepresentation produced above is not a permutative subrepresentation of the representation $\rho$ in the statement since neither of $\Phi_i$, $i=1,2$, is a basis vector.
\end{remark}

The last tool we need for the proof to the main result of the section is provided by the following lemma.

\begin{lemma}\label{pureinvariant}
Let $\rho:\CQ_2\rightarrow\CB(\CH)$ be a representation in which $\rho(S_1)$ and $\rho(S_2)$ are pure isometries.
If $M\subset\CH$ is a $\rho(\CO_2)$-invariant subspace, then it is also invariant under $\rho(\CQ_2)$.
\end{lemma}
\begin{proof}
All we have to do is show $\rho(U)M\subset M$ and $\rho(U^*)M\subset M$. We only deal with the first inclusion, since the second inclusion can be handled analogously. The equality $\rho(U)\rho(S_2)\rho(S_2)^*=\rho(S_1)\rho(S_2)^*$ says that
$\rho(S_2S_2^*)M$ is certainly invariant under $\rho(U)$. As for $\rho(S_1S_1^*)M$, we need to make use of the projections
$\rho(S_1S_2^n(S_2^*)^nS_1^*)$, $n\in \IN$, whose ranges decompose $\rho(S_1S_1^*)\CH$ into a direct sum as a consequence of $\rho(S_2)$ being pure.
 From the equality $\rho(S_2)\rho(U)\rho(S_2^n(S_2^*)^nS_1^*)=\rho(U)\rho(S_1S_2^n(S_2^*)^nS_1^*)$ we then see that
$\rho(U)\rho(S_1S_2^n(S_2^*)^nS_1^*)M$ is contained in $\rho(S_2)M$, which clearly concludes the proof.
\end{proof}

\begin{remark}
Actually, the proof above also covers the slightly more general situation in which only the restriction of the isometries to  the invariant subspace $M$ are assumed to be pure, while the isometries are allowed not to be 
so on the whole $\CH$.   
\end{remark}

We are finally in a position to prove the main result of this section, which not only says an irreducible permutative representation
of $\CQ_2$ is completely determined by its restriction to the Cuntz algebra $\CO_2$ but also provides the complete classification
of all irreducible permutative representations of $\CQ_2$:  as announced, apart from the canonical representation, every irreducible representation of $\CQ_2$ is the unique extension of an irreducible representation of $\CO_2$.

\begin{thm}\label{IrrRep}
If $\rho$ is an irreducible permutative representation of $\CQ_2$, then $\rho\upharpoonright_{\CO_2}$ is unitarily equivalent with either the restriction 
of the canonical representation to $\CO_2$ or an irreducible representation of type $P(I)$. In particular, $\rho\upharpoonright_{\CO_2}$ is multiplicity-free in
the sense of Bratteli-Jorgensen.
\end{thm}

\begin{proof}

If we denote the restriction of  $\rho$ to ${\CO_2}$ by $\pi$. we have $\pi=\oplus_{j\in J}\pi_j$, where each $\pi_j$ is a cyclic permutative representation of the
Cuntz algebra $\CO_2$. In particular, for every $j\in J$ there is a (possibly infinite) multi-index $I_j$ such that $\pi_j\cong P(I_j)$. If
a subrepresentation of type $P(1_k)$ shows up in the decomposition with a certain multiplicity, then a subrepresentation of
type $P(2_k)$ must show up as well with the same multiplicity, since this is the only way for $\sigma_p(\rho(S_1))$ and $\sigma_p(\rho(S_2))$ to coincide 
along with the multiplicity of each eigenvalue. Therefore, the decomposition of $\pi$ actually reads as 
$$\pi=n_1(P(1)\oplus P(2))\oplus n_2(P(11)\oplus P(22))\ldots\oplus n_k(P(1_k)\oplus P(2_k))\oplus\ldots\oplus\sigma$$
where $\sigma$ is the direct sum of the pure components. Now $\sigma$ uniquely extends to a representation
$\tilde{\sigma}$ of $\CQ_2$ by pureness, as does every $P(1_k)\oplus P(2_k)$ thanks to the theorem of Larsen and Li. 
In view of Proposition \ref{caninside} each $n_i$ must be $0$ for every $i\geq 2$, for otherwise $\rho$ would properly contain
a subrepresentation. In fact, $P(1)\oplus P(2)$ can appear. If it does, however, the pure part $\sigma$ cannot appear, hence 
the above direct sum is the canonical representation
up to equivalence.  
If it does not, then $\pi$ must be a pure representation of type $P(I)$, and
  we only need to show that $P(I)$ is already irreducible. But this is indeed a straightforward application of Lemma \ref{pureinvariant}.
\end{proof}

\begin{remark}
In light of the theorem we proved above, the canonical representation of $\CQ_2$ can now be characterized  as  its sole irreducible permutative representation whose restriction to $\CO_2$
is reducible. This raises the question of whether any (possibly non-permutative) irreducible representation of $\CQ_2$ which is not the canonical representation is still
irreducible when restricted to the Cuntz algebra $\CO_2$. One might also wonder to what extent an irreducible representation of $\CQ_2$ is determined by its restriction to $\CO_2$. 
%(up to equivalence, cf. Sect.3, the paragraph following Remark \ref{charpol}).

\end{remark}

To better appreciate the reach of Theorem \ref{IrrRep}, it is worth stressing that in general  a permutative representation of $\CQ_2$ will not restrict to the Cuntz algebra as a representation that is regular at all. For instance, this is seen by considering any of the permutative extensions of $P(1_k)\oplus P(2_k)$, where $k$ is any integer greater than $1$.  Indeed, none of these representations can be regular, since they do not decompose into a direct sum of irreducible permutative components, as we have already remarked.  
However, there is only one way for a restriction to fail to be regular: the representation $\rho$ itself must not decompose into permutative irreducibles at the level of $\CQ_2$. This is spelled out in the next proposition. Roughly speaking, it  says
one cannot jettison the hypothesis that $\rho\upharpoonright_{\CO_2}$ is regular if $\rho$ is to be decomposable at the level of $\CQ_2$.

\begin{proposition}\label{q2o2}
Let $\rho:\CQ_2\rightarrow\CB(\CH)$ be a permutative representation. The following are equivalent:
\begin{enumerate}
\item $\rho\upharpoonright_{\CO_2}$ is regular (multiplicity-free) in the sense of Bratteli-Jorgensen.
\item $\rho$ decomposes into the direct sum of (distinct) irreducible permutative subrepresentations.
\end{enumerate}
\end{proposition}
\begin{proof}
If $\pi\doteq \rho\upharpoonright_{\CO_2}$ is regular, then $\pi$ takes on the form $\pi=\oplus_{j\in J} P(I_j)$ 
%(also compare with Theorem \ref{decrep}), 
where
each $P(I_j)$ is irreducible. 
%and $P(I_j)\perp P(I_k)$ if $j\neq k$. 
Furthermore, if a representation of type $P(1)$ appears in the decomposition, then a representation of type $P(2)$ appears as well and with the same
multiplicity, so $\pi=n(P(1)\oplus P(2))\oplus(\oplus_{k} P(I_k))$, where the $P(I_k)$'s are all pure as well as irreducible, and $n\in\{0,1,2,\ldots,\infty\}$.
Therefore, $\rho$ must coincide with $n\rho_c\oplus(\oplus_k \widetilde P(I_k))$ up to unitary equivalence thanks to Theorem \ref{uniext}.
%But then $\rho$ must coincide  with $\rho_c\oplus (\oplus_{j\in J} \widetilde{P(I_j)})$, since each of the pure  
%components $P(I_j)$ extends uniquely as does $P(1)\oplus P(2)$ thanks to Remark \ref{howmanyext}.
On the other hand, assuming that $\rho$ decomposes as $\oplus_{j\in J} \rho_j$, where each $\rho_j$ is an irreducible permutative subrepresentation, 
immediately leads
to $\pi$ being a direct sum of a number (possibly zero) of copies of $\rho_c\upharpoonright_{\CO_2}$ and certain $P(I_j)$'s for suitable multi-indeces in view of Theorem
\ref{IrrRep}. Finally such a $\pi$ is regular by virtue of Corollary \ref{regular}.
To conclude, the same argument as above makes it plain that $\rho\upharpoonright_{\CO_2}$ is multiplicity-free if and only if $\rho$ is.
\end{proof}

%We would expect the above theorem to continue to hold if in (1) 'multiplicity-free' is replaced with 'regular' and in (2) 'distinct'  is simply removed.
%The implication (1) $\Rightarrow$ (2) still holds (if one wants to be pernickety we need to avail of Corollary \ref{uniext}).

Simple instances of representations other than those above arise quite naturally by composing the canonical representation of $\CQ_2$ with the endomorphisms $\chi_{2k+1}$ given by
$\chi_{2k+1}(S_2)=S_2$ and $\chi_{2k+1}(U)=U^{2k+1}$, $k\in\IZ$, introduced in \cite{ACR}. If we do so, we obtain a family
of representations  $\rho_{2k+1}\doteq \rho_c\circ\chi_{2k+1}$ of $\CQ_2$ acting on the Hilbert space $\ell^2(\IZ)$ as
$\rho_{2k+1}(S_2)e_l=e_{2l}$ and $\rho_{2k+1}(U)e_l=e_{l+2k+1}$, for every $l\in\IZ$.
For example, $\rho_3$ is the direct sum of the canonical representation with $\widetilde P(\overline{12})$. 
Among other things, this allows us to see that the multiplicity of $U$ in a (reducible) representation of type $\widetilde P(\overline{12})$ is two.
In general, every $\rho_{2k+1}$ admits a similar decomposition into a direct sum of two irreducible
components, one of which is always the canonical representation and the other is a representation of type
$P(I)$, with $I$ being a periodic infinite multi-index. \\

At the end of Section \ref{general} we pointed out that for any permutative representation $\rho$ of $\CQ_2$ the von Neumann algebra generated by $C^*(U)$ is never maximal unless
$\rho$ is the canonical representation. In fact, the diagonal subalgebra $\CD_2\subset\CO_2\subset\CQ_2$, which in \cite{ACR} was proved to be maximal abelian in $\CQ_2$ as well as in $\CO_2$,  
behaves in a dramatically different way: irrespective of what the irreducible permutative representation $\pi$ of either $\CO_2$ or  $\CQ_2$ is, the weak closure of $\pi(\CD_2)$  will  always be the atomic MASA $\ell^\infty(\IN)$. 
Actually, much more is true, for $\pi(\CD_2)''$ is the atomic MASA for every multiplicity-free representation $\pi$ of the Cuntz algebra.

\begin{proposition}
For any multiplicity-free representation $\pi:\CO_2\rightarrow B(\ell^2(\IN))$ the von Neumann algebra $\pi(\CD_2)''$ is $\ell^\infty(\IN)$. In particular, $\rho(\CD_2)''=\ell^\infty(\IN)$
for every irreducible representation $\rho$ of $\CQ_2$.
\end{proposition} 
\begin{proof}
For every $h\in \IN$ we denote the orthogonal projections onto $\IC e_h$ by $\delta_h$, i.e. $\delta_h(x)=(x, e_h)e_h$ for every $x\in\ell^2(\IN)$.
The statement is proved once we make sure each $\delta_h$ lies in the strong closure of $\pi(\CD_2)$.
Let $\alpha=(i_1,i_2,\ldots)\in \{1,2\}^\IN$ be the value of the coding map at $h$.  We denote by $\alpha_n$ the multi-index obtained 
out of $\alpha$ by taking the first $n$ entries only. The conclusion is reached if we show that $\pi(S_{\alpha_n}S_{\alpha_n}^*)$ strongly converges to $\delta_h$. 
To this aim, note that $\pi(S_{\alpha_n}S_{\alpha_n}^*)e_h=e_h$ for every $n\in\IN$ by definition. Furthermore, for every $k\neq h$ the sequence
$\pi(S_{\alpha_n}S_{\alpha_n}^*)e_k$ converges to zero in norm, for otherwise $\sigma(k)$ should be $\alpha$, which is not the case as $\sigma$ is injective by hypothesis. 
\end{proof}

\section{Pure states of $\CO_2$ with the unique extension property} \label{purestates}
The results we have obtained can also be reinterpreted in terms of extension of pure states. A pure state $\omega$ of the Cuntz algebra $\CO_2$ may have more
than one pure extension to $\CQ_2$. % give an example 
However, if $\omega$ comes from an irreducible permutative representation, then it will have precisely one pure extension, which is proved in the present section.
To begin with, we start by proving a general result that each pure state coming from an irreducible representations in which both Cuntz isometries are pure has a unique
pure extension to $\CQ_2$.

\begin{thm}
Every vector state associated with an irreducible representation $\pi:\CO_2\rightarrow \CB(\CH)$ in which
$\pi(S_1)$ and $\pi(S_2)$ are both pure has the unique extension property with respect to the inclusion $\CO_2\subset\CQ_2$.
\end{thm}

\begin{proof}
For a given $\omega_x$, with $\omega_x(T)=(\pi(T)x,x)$, for every $T\in \CO_2$, where $x\in \CH$ is a unit vector,
let $\Omega\in\CP(\CQ_2)$ be such that $\Omega\upharpoonright_{\CO_2}=\omega_x$.
Let $(\CH_\Omega, \pi_\Omega, x_\Omega)$
be the GNS triple associated with $\Omega$, that is $\pi_\Omega: \CQ_2\rightarrow \CB(\CH_\Omega)$ is the unique irreducible
representation such that $\Omega(T)=(\pi_\Omega(T)x_\Omega,x_\Omega)$ for every $T\in \CQ_2$.
Let us define $\CK\doteq \overline{\pi_\Omega(\CO_2)x_\Omega}\subset \CH_\Omega$. Then the representation 
$\pi':\CO_2\rightarrow\CB(\CK)$, which is given by $\pi'(T)\doteq \pi_\Omega(T)\upharpoonright_\CK$ for every $T\in\CO_2$,  is  cyclic by definition and the cyclic vector $x_\Omega$ satisfies $(\pi'(T)x_\Omega, x_\Omega)=\omega_x(T)$, for every $T\in \CO_2$, hence $\pi'$ and $\pi$ are unitarily equivalent.
In particular, $\pi_\Omega(S_i)\upharpoonright_\CK$ are both pure. But then $\CK$ is also invariant under $\pi_\Omega(\CQ_2)$ thanks to Lemma \ref{pureinvariant}, and so  $\CK=\CH_\Omega$ by irreducibility. In other terms, the restriction of $\pi_\Omega$ to the Cuntz algebra $\CO_2$ is nothing but
$\pi$ up to unitary equivalence. By pureness, though, $\pi$ can only be extended in one way, which means $\pi_\Omega$ is in fact
$\tilde{\pi}$, the unique extension of $\pi$ to $\CQ_2$, up to unitary equivalence. If $V:\CH_\Omega\rightarrow\CH$ is any intertwining unitary, i.e. $\pi_\Omega(T) = V^*\tilde\pi(T)V$ for every $T\in\CQ_2$, then
$\Omega(T)=(\pi_\Omega(T)x_\Omega, x_\Omega)=(\tilde\pi(T)Vx_\Omega, Vx_\Omega)$. The last equality says that $\Omega$ is uniquely determined, since
$Vx_\Omega$ must coincide up to a phase with $x\in \CH$.
\end{proof}

\begin{corollary}
Every vector state $\omega\in\CP(\CO_2)$ associated with an irreducible permutative representation of type $P(I)$
with $|I|\geq 2$ has the unique extension property relative to the inclusion $\CO_2\subset\CQ_2$.
\end{corollary}

Still, the representations of type $P(1)$ and $P(2)$ are out of the reach of the above result, but nevertheless their vector states continue to enjoy the unique extension property. 
This, however, is a consequence of the following simple and yet instrumental result.
\begin{lemma}
If $\rho:\CQ_2\rightarrow \CB(\CH)$ is a representation such that there exists an $\CO_2$-invariant subspace on which $\rho\upharpoonright_{\CO_2}$ acts as either
$P(1)$ or $P(2)$, then $\rho$ contains the canonical representation of $\CQ_2$.
\end{lemma}
\begin{proof}
We only deal with $P(2)$, since $P(1)$ can be dealt with in a similar fashion. 
By hypothesis, there is an $\CO_2$-invariant subspace $\CK(2)\cong \ell^2(\IN)$, with orthonormal basis
$\{e_n: n=0,1,2,\ldots\}$, such that $\rho(S_1)e_n=e_{2n+1}$ and $\rho(S_2)e_n=e_{2n}$, $n\in\IN\cup\{0\}$.
It is then straighforward to verify that $\rho(U)e_n=e_{n+1}$, for every $n\in\IN$. Furthermore, we also
have $\rho(U)e_0=e_1$. Indeed, $e_1=\rho(S_1)e_0=\rho(U)\rho(S_2)e_0=\rho(U)e_0$. It is now clear how to go on.
We define $e_{-n}\doteq \rho(U^*)^n e_0$ for every $k\in\IN$ and $\CK(1)\doteq\overline{\textrm{span}}\{e_{-n}:n\in \IN\}$. 
Clearly, $\{e_{-n}:n\in\IN\}$ is an orthonormal basis for $\CK(1)$. It is also as clear that $\CK(1)$ and $\CK(2)$ are orthogonal. Moreover, the equality $\rho(U)e_{-n}=e_{-n+1}$, for $n\in\IN$,
follows at once from the very definition of the vectors $e_{-n}$.
Now $\rho(S_2)e_{-n}=\rho(S_2)\rho(U^*)^ne_0=\rho(U^*)^{2n}\rho(S_2)e_0=e_{-2n}$ and $\rho(S_1)e_{-n}=\rho(U)\rho(S_2)e_{-n}=e_{-2n+1}$, for every
$n\in\IN$. This ends the proof, as $\CK(1)\oplus\CK(2)$ is the sought copy of the canonical representation of $\CQ_2$.
\end{proof}

\begin{thm}
Every vector state $\omega\in\CP(\CO_2)$ associated with either $P(1)$ or $P(2)$ has the unique extension property.
\end{thm}

\begin{proof}
We treat the case of a vector state $\omega_x$ coming from $P(2)$. If $\Omega\in\CP(\CQ_2)$ is any extension of $\omega_x$, we can consider its GNS triple $(\CH_\Omega, \pi_\Omega, x_\Omega)$.
We can then define $\CK\doteq \overline{\pi_\Omega(\CO_2)x_\Omega}$. By the uniqueness of the GNS triple, we see that $\CO_2$ acts on
$\CK$ as $P(2)$. But then $\pi_\Omega$ must contain a copy of the canonical representation and is in fact the canonical representation
by irreducibility. This shows that $\Omega$ is  a vector state rising from the canonical representation of $\CQ_2$ associated with a vector $x\in \overline{\textrm{span}}\{e_k:k\geq 0\}$, which concludes
our proof.
\end{proof}

Before leaving the section, we would like to point out that the problem of deciding whether a pure state has a unique
pure extension has mostly been settled in a context where a \emph{maximal abelian} subalgebra of a given
$C^*$-algebra is considered instead of any $C^*$-subalgebra. 
%Shall we briefly bring up the Kadison-Singer problem?  
A natural maximal abelian subalgebra of $\CQ_2$ that immediately springs to mind is of course
$C^*(U)$, not least because the diagonal $\CD_2$ has already been given a good deal of attention  in \cite{CuntzAut} % aggiungere riferimento a Cuntz
, where it is looked at in relation to the inclusion in the Cuntz algebra $\CO_2$.  
Obviously, the pure states of $C^*(U)$ are nothing but the evaluations at the points of the spectrum of $U$, namely 
the states $\omega_z$, $z\in\IT$,  given by $\omega_z(f(U))\doteq f(z)$ for every $f\in\ C(\IT)$.
It might come as a surprise that for some values of  $z$ the extension  of $\omega_z$ is unique whereas for others it is not. 
However, this closely resembles what happens with the inclusion $\CD_2\subset\CO_2$, \cite[Proposition 3.1]{CuntzAut}. 
We begin with the following result, which shows that roots of unity of order a power of $2$ give rise to as many extensions as possible.   
\begin{proposition}
If $z\in\IT$ is a root of unity of order $2^n$, for some $n\in\IN$, then the pure state $\omega_z$ has uncountably many
pure extensions to $\CQ_2$.
\end{proposition}
\begin{proof}
We start with the case  $z=1$. As already remarked, in the interval picture $U$ has eigenvalue $1$ corresponding to
the eigenfunction $\psi\in L^2([0,1])$, which is the function almost everywhere equal to $1$.
This clearly means that the vector state $\Omega$ given by $\Omega(T)=(T\psi, \psi)$, $T\in\CQ_2$, is a pure extension of
$\omega_1$. Composing $\Omega$ with the gauge automorphisms $\widetilde\alpha_w$  of $\CQ_2$, that is to say the unique extension to $\CQ_2$ of $\alpha_w\in\textrm{Aut}(\CO_2)$, $w\in\IT$,  we obtain an uncountable  family of pure
states $\Omega\circ\widetilde\alpha_w$,   $w\in\IT$, which all extend $\omega_1$ since $\tilde\alpha_w(U)=U$ for every $w\in\IT$. Finally, these are all distinct because
$\Omega\circ\widetilde\alpha_w (S_2)=w\Omega(S_2)$ and $\Omega(S_2)$ is different from zero.  
We can now move on to the more general situation in which $z$ is  a root of unity of order $2^n$, for some $n\in\IN$.
In this case there exists a unitary $U_z\in\CD_2\subset\CQ_2$ such that $U_zUU_z^*=zU$, see \cite{ACR} for more details. Therefore, the composition
$\Omega\circ\textrm{Ad}(U_z)\doteq\Omega_z$ is clearly a pure extension of $\omega_z$. However, we need to show there are in fact uncountably
many extensions of $\omega_z$. Again, these can be produced by composing $\Omega_z$ with the gauge automorphisms $\widetilde\alpha_w$, $w\in\IT$.
All is left to do, therefore, is make sure the compositions $\Omega_z\circ\widetilde\alpha_w$ are all distinct as $w$ ranges in $\IT$.
Now  the equality $\Omega_z\circ\widetilde\alpha_w=\Omega\circ\textrm{Ad}(U_z)\circ\widetilde\alpha_w=\Omega\circ\widetilde\alpha_w\circ\textrm{Ad}(U_z)$, where we used the fact
that $\widetilde\alpha_w$ and $\textrm{Ad}(U_z)$ commute as $\widetilde\alpha_w\circ\textrm{Ad}(U_z)\circ\widetilde\alpha_w^{-1}=\textrm{Ad}(\widetilde\alpha_w(U_z))=\textrm{Ad}(U_z)$,
shows that $\Omega_z\circ\widetilde\alpha_w=\Omega_z\circ\widetilde\alpha_{w'}$ holds if and only if $\Omega\circ\widetilde\alpha_w=\Omega\circ\widetilde\alpha_{w'}$, which is possible only
when $w=w'$.
\end{proof}
At the other extreme, when $z$ is not a root of unity of order $(2^h-1)2^k$, the corresponding state $\omega_z$ has precisely
one extension instead.
\begin{proposition}
If $z\in\IN$ is not a root of unity of order $(2^h-1)2^k$ for any $h,k\in\IN$, then the pure state $\omega_z$ has a unique pure
extension to $\CQ_2$.
\end{proposition}
\begin{proof}
Let $\Omega_z$ be an extension of $\omega_z$. We want to show that $\Omega_z$ is completely determined on a dense subset of $\CQ_2$. Thanks to \cite[Prop. 2.3]{ACR} it is enough to evaluate $\Omega_z$ on the elements of the form  $S_\alpha S_\beta^* U^h$. By using the GNS construction we may suppose that $\Omega_z(T)=(Tv,v)$, where $v$ is a vector such that $Uv=zv$.
First of all we have that $\Omega_z(S_2^k)=0$ for all $k\in\IN$. Indeed, we have $U^{2^k}(S_2^kv)=S_2^kUv=zS_2^kv$, while $U^{2^k}v=z^{2^k}v$. Under our assumptions $z^{2^k}\neq z$ and thus $S_2^kv$ and $v$ must be orthogonal. \\
We have that $\Omega_z(P_\alpha)=2^{-|\alpha|}$. This follows at once from the fact that $P_\alpha = U^iS_2^{|\alpha|}(S_2^*)^{|\alpha|}U^{-i}$ (for a certain $i$) and that $\sum_{i=0}^{2^k-1} U^iS_2^k(S_2^*)^kU^{-i}=1$.\\
Now let $h, k\in\IN$.
For the elements of the form $S_2^{h+k}(S_2^*)^k$ we have that 
\begin{align*}
\Omega_z(S_2^{h+k}(S_2^*)^k) & = (S_2^{h+k}(S_2^*)^k v,v) = (S_2^{k} v,v) -\sum_{i=1}^{2^k-1} (S_2^hU^iS_2^{k}(S_2^*)^k U^{-i}v,v)\\
& = -\sum_{i=1}^{2^k-1} \bar{z}^i(U^{2^hi}S_2^hS_2^{k}(S_2^*)^k v,v) = -\sum_{i=1}^{2^k-1} \bar{z}^i(S_2^hS_2^{k}(S_2^*)^k v,U^{-2^hi}v)\\
& = -\sum_{i=1}^{2^k-1} z^{2^hi-i}(S_2^hS_2^{k}(S_2^*)^k v,v)
\end{align*}
which leads to 
$$
(S_2^{h+k}(S_2^*)^k v,v) \left(\sum_{i=0}^{2^k-1} (z^{2^h-1})^i\right) =(S_2^{h+k}(S_2^*)^k v,v) \cdot \frac{1-z^{2^h2^k-2^k}}{1-z^{2^h-1}} =0
$$
Under our hypotheses it follows that  $\Omega_z(S_2^{h+k}(S_2^*)^k) = (S_2^{h+k}(S_2^*)^k v,v) =0$.\\
From this discussion we get that $\Omega_z(S_\alpha S_\beta^* U^h)=\delta_{|\alpha|,|\beta|}2^{-|\alpha|}z^h$ and we are done.
\end{proof}
Clearly, the case of a root of unity of order $(2^h-1)2^k$, for some $h$ and $k\in\IN$, is not covered by
the results above, although we would be inclined to believe $\omega_z$ does have a unique extension if $h\geq 2$.  
However, we plan to return to this problem elsewhere.

\section{Extendible quadratic permutation endomorphisms} \label{extendible}
This section aims to intertwine the analysis carried out in \cite{ACR}, where
we addressed the problem of extending endomorphisms of $\CO_2$ to $\CQ_2$, with the present analysis.
Although a satisfactory answer to the problem is yet to come and might be elusive to get to, particular classes of automorphisms, such as  Bogolubov automorphism and localized diagonal automorphisms, 
have been examined thoroughly.
%In our first two papers on the dyadic ring $C^*$-algebra \cite{ACR, ACR2} we studied the extension problem for endomorphism of $\CO_2$ to $\CQ_2$.   
Given a permutative representation $\rho$ of $\CQ_2$ and a permutative endomorphism $\lambda$ of $\CO_2$, the composition
$\rho\upharpoonright_{\CO_2}\circ\lambda$ is still  a permutative representation of $\CO_2$. Accordingly, one might ask whether it extends.
Furthermore, one may want to go so far as to ask that the extension be of the form $\rho\circ\tilde\lambda$, where
$\tilde\lambda$ is an endomorphism of $\CQ_2$ that extends $\lambda$.
In the sequel, we shall be dealing with the case in which $\rho=\rho_c$ is the canonical representation and $\lambda$ is a so-called quadratic
permutative endomorphism of $\CO_2$, that is an endomorphism induced by a permutation matrix in $\CF_2^2\cong \textrm{M}_4(\IC)$.
In this case, the extension $\rho_c\circ\tilde\lambda$ is automatically permutative when it exists. %, as it will be clearer later on.  
The quadratic permutative endomorphisms include the canonical endomorphism $\varphi$ of $\CO_2$ and the flip-flop $\lambda_f$. Both of them extend to $\CQ_2$, as proved in \cite{ACR}. More interestingly, this family of endomorphisms offers a bunch of novel examples of endomorphisms that do
extend to $\CQ_2$, although the list of the extendible endomorphisms is still rather limited.
It turns out that an effective way to prove that the classes of these endomorphisms are in fact distinct is to resort to permutative representations of $\CQ_2$. For example, we 
will show that the representations obtained by composing 
the canonical representation with two of the above new endomorphisms are inequivalent.

\medskip
We adopt the same notation as in \cite{CS09}, where the monomial $s_is_js_k^*$ is denoted by  $s_{ij,k}$. The following table displays not only the definitions of the endomorphisms  we are going to consider 
but also  says in advance which ones extend and which ones do not extend.

{\footnotesize
\[
\begin{array}{lcccc}
%\multicolumn{4}{c}{%\mbox{Table 1. Entropy and index of the `rank 2' permutation endomorphisms of $\CO_2$.}
%}\\
\\
\hline  \rho_{\sigma}& \rho_{\sigma}(s_{1})& \rho_{\sigma}(s_{2})& Extendible? & \rho_{\sigma}(u) \\
\hline
\rho_{id} = {\rm id} & s_{1} & s_{2}& {\rm Yes} & u\\
\rho_{12}   &s_{12,1}+s_{11,2}& s_{2}&  {\rm No} & \\
\rho_{13}   &s_{21,1}+s_{12,2}& s_{11,1}+s_{22,2} & {\rm No} &  \\
\rho_{14}   &s_{22,1}+s_{12,2}& s_{21,1}+s_{11,2}&  {\rm Yes} & u^{-2}\\
\rho_{23} (= \varphi)  &s_{11,1}+s_{21,2}& s_{12,1}+s_{22,2}&  {\rm Yes} & u^{2}\\
\rho_{24}   &s_{11,1}+s_{22,2}&s_{21,1}+s_{12,2}& {\rm No} & \\
\rho_{34}   &s_{1}& s_{22,1}+s_{21,2}&  {\rm No} &\\
\rho_{123} &s_{12,1}+s_{21,2}& s_{11,1}+s_{22,2}&  {\rm Yes} & u^2s_2s_2^*+u^{-2}s_1s_1^*\\
\rho_{132} &s_{21,1}+s_{11,2}& s_{12,1}+s_{22,2}& {\rm No} & \\
\rho_{124} &s_{12,1}+s_{22,2}& s_{21,1}+s_{11,2}& {\rm No} & \\
\rho_{142} &s_{22,1}+s_{11,2}& s_{21,1}+s_{12,2}& {\rm No} & \\
\rho_{134} (\simeq \rho_{142}) &s_{21,1}+s_{12,2}& s_{22,1}+s_{11,2}& {\rm No} & \\
\rho_{143} &s_{22,1}+s_{12,2}& s_{11,1}+s_{21,2}& {\rm No} & \\
\rho_{234} &s_{11,1}+s_{21,2}& s_{22,1}+s_{12,2}& {\rm No} & \\
\rho_{243} (\simeq \rho_{123}) &s_{11,1}+s_{22,2}& s_{12,1}+s_{21,2}& {\rm Yes} & u^{-2}s_2s_2^*+u^{2}s_1s_1^*\\
\rho_{1234} (\simeq \rho_{24}) &s_{12,1}+s_{21,2}& s_{22,1}+s_{11,2}& {\rm No} & \\
\rho_{1243} &s_{12,1}+s_{22,2}& s_{11,1}+s_{21,2}& {\rm Yes} & u^{-2}\\
\rho_{1324} (\simeq \rho_{12}) &s_{2}& s_{12,1}+s_{11,2}& {\rm No} & \\
\rho_{1342} &s_{21,1}+s_{11,2}& s_{22,1}+s_{12,2}& {\rm Yes} & u^2 \\
\rho_{1423} (\simeq \rho_{34}) &s_{22,1}+s_{21,2}& s_{1}& {\rm No} & \\
\rho_{1432} (\simeq \rho_{13}) &s_{22,1}+s_{11,2}& s_{12,1}+s_{21,2}& {\rm No} & \\
\rho_{(12)(34)} (\simeq \rho_{(13)(24)}) &s_{12,1}+s_{11,2}& s_{22,1}+s_{21,2}& {\rm Yes} & fu^* f\\
\rho_{(13)(24)}=\lambda_f &s_{2}& s_{1}&  {\rm Yes} & u^*\\
\rho_{(14)(23)} (\simeq {\rm id}) &s_{22,1}+s_{21,2}& s_{12,1}+s_{11,2}& {\rm Yes} & fuf\\
\hline
\end{array}
\]
}

We start with $\rho_{12}$. %Consider the proper endomorphism $\rho_{12}: \CO_2\to \CO_2$ defined by $\rho_{12}(S_2)=S_2$, 
Note that $\rho_{12}(S_1)=S_1f$, where $f=S_1S_2^* +S_2S_1^*\in \CU(\CO_2)$. This observation will be crucial in the proof of the following result.

\begin{proposition}
The endomorphism $\rho_{12}$ does not extend.
\end{proposition}
\begin{proof}
First of all we identify $\CQ_2$ with the image of the canonical representation $\rho_c: \CQ_2\to \CB(\ell^2(\IZ))$.
We set $\tilde S_i = \rho_{12}(S_i)$ and we denote by $\tilde U$ the element in $\CB(\ell^2(\IZ))$ that extends the representation $\rho_c\circ \rho_{12}$.
It is easy to see that if $\rho_{12}$ is extendible, then the extension is unique. This implies that $\tilde U\in \CQ_2^\IT$.
Since $S_2^k \tilde U = \tilde U \tilde S_1^k$ for all $k\in\IN$, we have that
$$
S_2^k \tilde U = \tilde U  (S_1f)^k = \tilde U  S_1 S_2^{k-1}f=\tilde U  U S_2^{k}f
$$
and thus $\tilde U = (S_2^*)^k \tilde U U S_2^k f$. 
From \cite[Proposition 3.18]{ACR} we know that the following limit exists 
$$
\lim_k (S_2^*)^k \tilde U U S_2^k = c_{\tilde U U}\in \IT
$$
Therefore, we have $\tilde U =  c_{\tilde U U} f$. It is easy to see that this unitary does not satisfy the defining relations of $\CQ_2$.
 \end{proof}
%Despite the previous result, the representation $\pi\circ\alpha$ does extend to a representation of $\CQ_2$. In fact, the unitary parts of the Wold decomposition of $\pi(\alpha(S_1))$ and $\pi(\alpha(S_1))$ are  immediately seen to be unitarily equivalent, being equal to $1\in \IC$ because both operators have a $1$-dimensional eigenspace corresponding to the eigenvalue $1$, cf. \cite{LarsenLi}. 
%Moreover, the formula of $\pi(\alpha(U))$ is explicitly given in Proposition \ref{formula-alpha-U}.

\begin{proposition}
%Henceforward we adopt the same notations as in  \cite{CS09}. 
The endomorphisms $\rho_{34}$, $\rho_{1324}$ and $\rho_{1423}$ do not extend. 
\end{proposition} 
\begin{proof}
As for the first, it is enough to rewrite it as $\rho_{34}=\lambda_f\circ \rho_{12}\circ \lambda_f$, whereas the last two are unitarily equivalent to $\rho_{12}$ and $\rho_{34}$ respectively, as pointed out in \cite{CS09}.
\end{proof}
\begin{proposition}
The endomorphism $\rho_{13}$ does not extend. 
\end{proposition}
\begin{proof}
By definition $\rho_{13}(S_1)\doteq S_2S_1S_1^*+S_1S_2S_2^*$, $\rho_{13}(S_2)\doteq S_1^2S_1^*+S_2^2S_2^*$. It is now a matter of straightforward computations that, in the canonical representation, the point spectrum
of $\rho_{13}(S_1)$ is empty whereas that of $\rho_{13}(S_2)$ is not, since it is the set $\{1\}$ with $\ker (\rho_{13}(S_2)-I)={\rm span}\{e_{-1}, e_0\}$.
 \end{proof}
\begin{proposition}
None of the endomorphisms $\rho_{1432}$, $\rho_{24}$, and $\rho_{1234}$, extend. 
\end{proposition}
\begin{proof}
The first is unitarily equivalent to $\rho_{13}$, see \cite{CS09}. The second may be
rewritten as $\rho_{24}=\lambda_f \circ \rho_{13} \circ \lambda_f$. Finally, $\rho_{1234}$ does not extend, since it is unitarily equivalent to $\rho_{24}$. 
\end{proof}

\begin{proposition}
None of endomorphisms $\rho_{132}$, $\rho_{234}$, $\rho_{124}$ and $\rho_{143}$  extend. 
\end{proposition}
\begin{proof}
By definition $\rho_{132}(S_1)\doteq S_2S_1S_1^*+S_1^2S_2^*$, $\rho_{132}(S_2)\doteq S_1S_2S_1^*+S_2^2S_2^*$. The conclusion is immediately arrived at by noting that $\ker (\rho_{132}(S_2)-I)={\rm span}\{e_{0}, e_1\}$ whilst $\ker (\rho_{132}(S_1)-I)=\{0\}$.  The endomorphism $\rho_{234}$ is not extendible because $\rho_{132}=\lambda_f\circ \rho_{234}\circ\lambda_f$.  Finally, being $\rho_{124}=\lambda_f\circ \rho_{234}$ and $\rho_{143}=\rho_{124}\circ\lambda_f$, $\rho_{124}$ and $\rho_{143}$ are not extendible either.
 \end{proof}

The next proposition recalls those endomorphisms that we already know do extend, without us needing to give any explanation, since the missing details are to be found in \cite{ACR}.

\begin{proposition}
The automorphisms $\rho_{(13)(24)}=\lambda_f$, $\rho_{(14)(23)}=\Ad(f)$, $\rho_{(12)(34)}=\Ad(f)\circ\lambda_f$ all extend, as does the canonical
endomorphism $\varphi=\rho_{23}$. Moreover, they send $U$ to $U^*$, $fUf$, $fU^* f$ and $U^2$, respectively.
\end{proposition}

We can now move on to discuss the novel examples of extendible endomorphisms. We start with $\rho_{123}$ and $\rho_{243}$.

\begin{proposition}
Both $\rho_{123}$ and $\rho_{243}$ uniquely extend. Moreover, $\rho_{123}(U)=U^2S_2S_2^*+U^{-2}S_1S_1^*$
and $\rho_{243}(U)=U^{-2}S_2S_2^*+U^2S_1S_1^*$.
\end{proposition}
\begin{proof}
Since $\rho_{243}=\rho_{123}\circ\lambda_f$, it is enough to deal only with $\rho_{123}$. 
By definition $\rho_{123}(S_1)\doteq S_1S_2S_1^*+S_2S_1S_2^*\doteq \tilde S_1$ and $\rho_{123}(S_2)\doteq S_1^2S_1^*+S_2^2S_2^*\doteq \tilde S_2$. 
We define $\tilde U\doteq U^2S_2S_2^*+U^{-2}S_1S_1^*$. All we have to do is make sure the two equalities
$\tilde U\tilde S_2=\tilde S_1$ and $\tilde S_2\tilde U=\tilde U\tilde S_1$ are satisfied. Now this is a matter of easy but tedious computations.
Finally, the uniqueness of the extension is proved once a representation  $\rho$ of $\CQ_2$ is exhibited in which either $\tilde S_1$ or $\tilde S_2$ is pure.
For instance, the interval picture is an example of such a representation.
\end{proof}

The endomorphism $\rho_{14}$ extends as well. This is the content of the following proposition, which is proved in full detail instead.
\begin{proposition}
The endomorphism $\rho_{14}$ is extendible. Moreover, its unique extension is determined by $\rho_{14}(U)\doteq U^{-2}$.
\end{proposition}
\begin{proof}
Set $\tilde S_1 =\rho_{14}(S_1)=S_2^2S_1^*+S_1S_2S_2^*$, $\tilde S_2=\rho_{14}(S_2)= S_2S_1S_1^*+S_1^2S_2^*$, and $\tilde U=U^{-2}$. We recall that $U^{-2}S_iS_i^*=S_iS_i^*U^{-2}$ for $i=1, 2$. We have that
\begin{align*}
\tilde U \tilde S_2 & = U^{-2}S_2S_1S_1^*+U^{-2}S_1^2S_2^*\\
& = U^{-2}S_2US_2S_1^*+U^{-2}US_2US_2S_2^*\\
& = U^{-2}U^2S_2S_2S_1^*+U^{-2}UU^2S_2S_2S_2^*\\
& = S_2^2S_1^*+US_2^2S_2^*\\
& = S_2^2S_1^*+S_1S_2S_2^*=\tilde S_1
\end{align*}
and 
\begin{align*}
\tilde S_2 \tilde U &= S_2S_1S_1^*U^{-2}+S_1^2S_2^*U^{-2}\\
&= S_2U^{-2}S_1S_1^*+S_1US_2S_2^*U^{-2}\\
&= S_2U^{-2}US_2S_1^*+S_1UU^{-2}S_2S_2^*\\
&= S_2U^*S_2S_1^*+S_1U^*S_2S_2^*\\
&= U^{-2}S_2S_2S_1^*+U^{-2}S_1S_2S_2^*\\
&= U^{-2}\tilde S_1\; .
\end{align*}
Therefore, by universality an extension exists. 
In order to prove that the extension is unique, it is enough to exhibit a representation $\rho$ of $\CQ_2$ such that $\rho(\tilde S_2)$ is pure. 
First of all, we note that $\tilde S_2^k(\tilde S_2^*)^k=S_2S_1^{k}(S_1^*)^{k}S_2^*+S_1^{k+1}(S_1^*)^{k+1}$ for all $k\geq 1$.
Now the claim follows at once by using the interval picture and the fact the evaluation of $\CF_2$-trace on $\tilde S_2^k(\tilde S_2^*)^k$ is $2^{-k}$. 
%The extension is unique because the isometries $\tilde S_1$ and $\tilde S_2$ are both pure.
\end{proof}
\begin{proposition}
The endomorphisms $\rho_{1243}$ and $\rho_{1342}$ both extend, mapping $U$ into $U^{-2}$ and $U^2$, respectively.
\end{proposition}
\begin{proof}
Both $\rho_{1243}$ and $\rho_{1342}$ commute with the flip-flop. More precisely, we have 
$\lambda_f\circ\rho_{1243}=\rho_{1243}\circ\lambda_f=\rho_{23}$
and $\lambda_f\circ\rho_{1342}=\rho_{1342}\circ\lambda_f=\rho_{14}$. %As a result, 
\end{proof}

\begin{proposition}
The representations $\pi_1\doteq \rho_c\circ \rho_{23}=\rho_c\circ \varphi$ and $\pi_2\doteq\rho_c\circ \rho_{1243}$ are both permutative and equivalent to %the extension of $P(1)\oplus P(1)\oplus P(2) \oplus P(2)$.
$\rho_c\oplus \rho_c$.
\end{proposition}
\begin{proof}
Thanks to the previous result, it is enough to prove the statement for $\pi_1$. We recall that $\rho_{23}$ is the canonical endomorphism, that is $\rho_{23}(x)=S_1xS_1^*+S_2xS_2^*$. Simple computations lead to these formulas
\begin{align*}
& \pi_1(S_2) e_{2k} = e_{4k}\\
& \pi_1(S_2) e_{2k+1} = e_{4k+1}\\
& \pi_1(S_1) e_{2k} = e_{4k+2}\\
& \pi_1(S_1) e_{2k+1} = e_{4k+3}\; .
\end{align*}
In particular, we find the equalities $\pi_1(S_1)e_{-1} = e_{-1}$, $\pi_1(S_2)e_{0} = e_{0}$, $\pi_1(S_1)e_{-2} = e_{-2}$, $\pi_1(S_2)e_{1} = e_{1}$. Therefore, the claim follows once we realize that 
the direct sum $\overline{\pi_1(\CO_2)e_{-1}}\oplus \overline{\pi_1(\CO_2)e_{-1}}\oplus \overline{\pi_1(\CO_2)e_{0}}\oplus \overline{\pi_1(\CO_2)e_{-2}}$ is the whole $\ell^2(\IZ)$.
\end{proof}

%=\rho_{1342}\circ\lambda_f=\rho_{14}$. 
\begin{proposition}
The representations $\pi_1\doteq \rho_c\circ \rho_{14}$ and $\pi_2\doteq\rho_c\circ \rho_{1342}$ are both permutative and equivalent to $\tilde P(11)\oplus \tilde P(22)$. %$=\tilde P(1)\oplus \tilde P(1,-1)\oplus \tilde P(2)\oplus \tilde P(2,-1)$, that is $\rho_c\oplus(\rho_c\circ \tilde\alpha_{-1})$. 
In particular, $\rho_{14}$ cannot be obtained from $\varphi=\rho_{23}$ by composing with inner automorphisms and the flip-flop.
\end{proposition}
\begin{proof}
It is enough to prove the statement for $\pi_1$. We have that
\begin{align*}
& \pi_1(S_2) e_{2k} = e_{4k+3}\\
& \pi_1(S_2) e_{2k+1} = e_{4k+2}\\
& \pi_1(S_1) e_{2k} = e_{4k+1}\\
& \pi_1(S_1) e_{2k+1} = e_{4k}\; .
\end{align*}
It follows that $\pi_1(S_2^2)e_{-1} = \pi_1(S_2)e_{-2}=e_{-1}$ and $\pi_1(S_1^2)e_0=\pi_1(S_1)e_1=e_0$.  
Simple computations show that $\{e_k: k\leq -1\}\subset \pi_1(\CO_2)e_{-1}$ and $\{e_k: k\geq 0\}\subset \pi_1(\CO_2)e_{0}$. Now the last claim follows at once from the decomposition of $\rho_c\circ \rho_{23}$ shown in the previous proposition, and the fact that $P(11)\oplus  P(22)$ is not regular,  see Proposition \ref{q2o2}.
\end{proof}
%In the next proposition we complete our analysis of the endomorphisms considered in  \cite{CS09}.
To complete our analysis, all is left to do is deal with the endomorphisms $\rho_{134}$ and $\rho_{142}$, which are defined by these formulas
%In the first place, we recall that %$\rho_{134}(S_1)= S_2S_1S_1^*+S_1S_2S_2^*$, $\rho_{134}(S_2)= S_2^2S_1^*+S_1^2S_2^*$, and $\rho_{142}(S_1)= S_2^2S_1^*+S_1^2S_2^*$, $\rho_{142}(S_2)= S_2S_1S_1^*+S_1 S_2S_2^*$.
\begin{align*}
&\rho_{134}(S_1)= S_2S_1S_1^*+S_1S_2S_2^*\\
&\rho_{134}(S_2)= S_2^2S_1^*+S_1^2S_2^*\\
&\rho_{142}(S_1)= S_2^2S_1^*+S_1^2S_2^*\\
&\rho_{142}(S_2)= S_2S_1S_1^*+S_1 S_2S_2^*
\end{align*} 
%We observe that $\rho_{134}(S_1)=\rho_{13}(S_1)$ and $\rho_{134}(S_2)=\rho_{1234}(S_2)$, $\rho_{142}(S_1)=\rho_{13}(S_1)$ and $\rho_{142}(S_2)=\rho_{24}(S_2)$ (all the endomorphisms $\rho_{13}$, $\rho_{123}$, $\rho_{1234}$, $\rho_{24}$ are non-extendible). 

\begin{proposition}
The endomorphisms $\rho_{134}$ and $\rho_{142}$ are not extendible.
\end{proposition}
It is known that $\rho_{134}\simeq \rho_{142}$, so it is enough to prove that $\rho_{134}$ does not extend. The proof employs a strategy similar to the one used in \cite{ACR} to determine the extendible Bogolubov automorphisms, but the argument is actually more complicated and quite long. For the sake of clarity, our proof is thus divided into a series of preliminary lemmas.
As usual we identify $\CQ_2$ with its image in the canonical representation.
 %so we'll divide it into a series of lemmas. 
%We still have to decide whether $\rho_{134}$ and $\rho_{142}\simeq \rho_{134} $ extend or not.

\begin{lemma}
In the canonical representation we have that 
\begin{align*}
&\rho_{134}(S_1)e_{2k}=e_{4k+1}\\
&\rho_{134}(S_1)e_{2k+1}=e_{4k+2}\\
&\rho_{134}(S_2)e_{2k}=e_{4k+3}\\
&\rho_{134}(S_2)e_{2k+1}=e_{4k}
\end{align*} 
%$\rho_{134}(S_1)e_{2k}=e_{4k+1}$, $\rho_{134}(S_1)e_{2k+1}=e_{4k+2}$, $\rho_{134}(S_2)e_{2k}=e_{4k+3}$, $\rho_{134}(S_2)e_{2k+1}=e_{4k}$. 
In particular, the two isometries $\rho_{134}(S_1)$ and $\rho_{134}(S_2)$ are pure. % and the point spectrum is empty.
\end{lemma} 
\begin{lemma}
If $\rho_{134}$ is extendible, then the extension is unique. Moreover, in this case $\rho_{134}(U)\in\CQ_2^\IT$.
\end{lemma}
\begin{proof}
The first part is an immediate consequence of the pureness of $\rho_{134}(S_1)$ and $\rho_{134}(S_2)$ in the canonical representation. The second part can be proved with the same strategy used in \cite[Lemma 4.10]{ACR}.
\end{proof}
%\begin{lemma}
%If $\rho_{134}$ is extendible, then $\lambda_f(\rho_{134}(U))=\rho_{134}(U)$, where $\lambda_f$ is the flip-flop automorphism. In particular $\rho_{134}(U)\neq U^k$ for all $k\in\IZ$.
%\end{lemma}
%\begin{proof}
%We have that $\lambda_f(\rho_{134}(S_i))=\rho_{134}(S_i)$ for $i=1, 2$. It is easy to see that $\lambda_f(\rho_{134}(U))$ satisfies the same relations as $\rho_{134}(U)$. The claim follows from the fact that there is at most  one extension of $\rho_{134}$ to $\CQ_2$.
%\end{proof}
%\begin{lemma}
%If $\rho_{134}$ extends, then $\rho_{134}(U)\in\CQ_2^\IT$.
%\end{lemma}
%\begin{proof}
%The claim can be proved with the same strategy used in \cite[Lemma 4.10, p. 70]{ACR}.
%\end{proof}
\begin{lemma} \label{L10.19}
If $\rho_{134}$ is extendible, we have that $$\rho_{134}(U)S_2^2(S_2^*)^2=U^2S_2^2(S_2^*)^2,  \qquad \rho_{134}(U)S_1^2(S_1^*)^2=U^{-2}S_1^2(S_1^*)^2 \ . $$
\end{lemma}
\begin{proof}
By evaluating the equality $\rho_{134}(U)\rho_{134}(S_2)=\rho_{134}(S_1)$ at $e_{2k}$ and $e_{2k+1}$ we get that $\rho_{134}(U)e_{4k}=e_{4k+2}$ and $\rho_{134}(U)e_{4k+3}=e_{4k+1}$. Now the claim follows.
\end{proof}
\begin{lemma}
We have that
\begin{align*}
&\rho_{134}(S_1)^{2k}=\tilde S_1^{2k}=S_2(S_1S_2)^kS_2^*+S_1(S_2S_1)^kS_1^*\\
&\rho_{134}(S_2)^{2k}=\tilde S_2^{2k}=S_2(S_2S_1)^kS_2^*+S_1(S_1S_2)^kS_1^*
\end{align*} 
\end{lemma}
\begin{proof}
These equalities can be proved by induction. When $k=1$, we have that
\begin{align*}
\rho_{134}(S_1)^{2} & = (S_2S_1S_1^*+S_1S_2S_2^*) (S_2S_1S_1^*+S_1S_2S_2^*)\\
& =  S_2(S_1S_2)S_2^*+S_1(S_2S_1)S_1^*
\end{align*} 
and
\begin{align*}
\rho_{142}(S_2)^2 & = (S_2^2S_1^*+S_1^2S_2^*) (S_2^2S_1^*+S_1^2S_2^*)\\
& = S_2(S_2S_1)S_2^*+S_1(S_1S_2)S_1^* \; .
\end{align*} 
Supposing that the formulas hold for $k$, we now prove that they are true for $k+1$ too. We have that
\begin{align*}
\rho_{134}(S_1)^{2(k+1)} &= \rho_{134}(S_1)^{2k}\rho_{134}(S_1)^{2}\\
& = (S_2(S_1S_2)^kS_2^*+S_1(S_2S_1)^kS_1^*)(S_2(S_1S_2)S_2^*+S_1(S_2S_1)S_1^*)\\
& = S_2(S_1S_2)^{+1}kS_2^*+S_1(S_2S_1)^{k+1}S_1^*
\end{align*} 
and
\begin{align*}
\rho_{134}(S_2)^{2(k+1)} &= \rho_{134}(S_2)^{2k}\rho_{134}(S_2)^{2}\\
& = (S_2(S_2S_1)^kS_2^*+S_1(S_1S_2)^kS_1^*)(S_2(S_2S_1)S_2^*+S_1(S_1S_2)S_1^*)\\
& = S_2(S_2S_1)^{k+1}S_2^*+S_1(S_1S_2)^{k+1}S_1^*
\end{align*} 
\end{proof}
By using the defining relations of $U$ and $S_2$ we see that $US_1^k=S_2^kU$. This leads to 
\begin{align*}
(\tilde S_2^*)^{2k} \tilde U \tilde S_1^{2k} & = (S_2(S_2S_1)^kS_2^*+S_1(S_1S_2)^kS_1^*)^* \tilde U (S_2(S_1S_2)^kS_2^*+S_1(S_2S_1)^kS_1^*)\\
& = (S_2(S_1^*S_2^*)^kS_2^*+S_1(S_2^*S_1^*)^kS_1^*) \tilde U (S_2(S_1S_2)^kS_2^*+S_1(S_2S_1)^kS_1^*)\\
& = S_2[(S_1^* S_2^*)^kS_2^* \tilde U S_2(S_1S_2)^k]S_2^*+S_2[(S_1^* S_2^*)^k S_2^* \tilde U S_1(S_2S_1)^k]S_1^*\\
& +S_1[(S_2^* S_1^*)^k S_1^*\tilde U S_2(S_1S_2)^k]S_2^*+S_1[(S_2^* S_1^*)^kS_1^* \tilde U S_1(S_2S_1)^k]S_1^* = \tilde U\; .
\end{align*} 
If we set
\begin{align*}
&A_k\doteq (S_1^* S_2^*)^kS_2^* \tilde U S_2(S_1S_2)^k	\\
&B_k\doteq (S_1^* S_2^*)^k S_2^* \tilde U S_1(S_2S_1)^k\\
&C_k\doteq (S_2^* S_1^*)^k S_1^*\tilde U S_2(S_1S_2)^k\\
&D_k \doteq (S_2^* S_1^*)^kS_1^* \tilde U S_1(S_2S_1)^k 
\end{align*} 
we see that the sequence $z_k\doteq S_2A_kS_2^*+S_2B_kS_1^*+S_1C_kS_2^*+S_1D_kS_1^*$ tends to $\tilde U$. % (actually it is constant). 
Note that $A_k, B_k, C_k, D_k\in \CQ_2^\IT$.
The following result is an easy but useful fact which can be obtained by multiplying $z_k$ on the left by $S_i^*$ and on the right by $S_j$ for appropriate $i$ and $j$.
\begin{lemma}
The sequences $A_k$, $B_k$, $C_k$, and $D_k$ are convergent and the limits belong to $\CQ_2^\IT$.
\end{lemma}
As in \cite[Section 3.3, p. 62]{ACR} we set $\CB^k_2\doteq {\rm span}\{S_\alpha S_\beta U^h \; | \; |\alpha |=|\beta | = k,  h\in\IZ\}$.
\begin{lemma}
If $x\in \CB_2^k$, $i_1,j_1,i_2,j_2 \in \{1,2\}$, $i_1\neq j_1$, $i_2\neq j_2$, then $(S_{i_1}^* S_{j_1}^*)^k x(S_{i_2}S_{j_2})^k\in \IC[U]$. 
\end{lemma}
\begin{lemma}(cf. \cite[Lemma 4.12, p. 70]{ACR})
Let $x\in \CQ_2^\IT$, $i_1\neq j_1$, $i_2\neq j_2$ such that the sequence $(S_{i_1}^* S_{j_1}^*)^k x(S_{i_2}S_{j_2})^k$ converges to an element $z$. Then $z\in C^*(U)$.
\end{lemma}
\begin{proof}
Let $\{y_k\}_{k\geq 0}$ be a sequence such that $y_k\in\CB_2^k$ and $y_k\to x$ normwise. Then the thesis follows from the inequality
$$
\| z- (S_{i_1}^* S_{j_1}^*)^k y_k(S_{i_2}S_{j_2})^k\|\leq \| z-(S_{i_1}^* S_{j_1}^*)^k x(S_{i_2}S_{j_2})^k\|+ \| (S_{i_1}^* S_{j_1}^*)^k (x-y_k)(S_{i_2}S_{j_2})^k\|\; .
$$
\end{proof}
By an immediate application of the previous lemma we get the following result.
\begin{lemma}
We have that
\begin{align*}
&\lim_k A_k=f_1(U)\; , \quad \lim_k B_k=f_2(U)\; , \quad \lim_k C_k=f_3(U)\; , \quad \lim_k D_k=f_4(U)
\end{align*} 
where $f_1, f_2, f_3, f_4\in C(\IT)$. 
\end{lemma}
This shows that, %Thanks to the previous result we  know that,
if $\rho_{134}$ is extendible, then $\rho_{143}(U)$ has the following form
\begin{align*}
\rho_{143}(U) & = \lim_k S_2 A_kS_2^* + S_2 B_k S_1^* +S_1 C_k S_2^*+ S_1D_kS_1^*\\
& = S_2 f_1(U)S_2^* + S_2 f_2(U) S_1^* +S_1 f_3(U)S_2^*+ S_1f_4(U)S_1^*\\ 
& = f_1(U^2) S_2 S_2^* + f_2(U^2)S_2 S_1^* +f_3(U^2)S_1S_2^*+ f_4(U^2)S_1S_1^*\\ 
& = (f_1(U^2)+f_3(U^2)U) P_2  + (f_2(U^2)U^*+ f_4(U^2))P_1\\ 
& = (f_1(U^2)+f_3(U^2)U) P_{22}+(f_1(U^2)+f_3(U^2)U) P_{21}\\
&+(f_2(U^2)U^*+ f_4(U^2))P_{11}+(f_2(U^2)U^*+ f_4(U^2))P_{12}
\end{align*} 
By Lemma \ref{L10.19} we also know that $\tilde U = U^2 P_{22}+U^{-2}P_{11}+\tilde U P_{21}+ \tilde U P_{12}$. If we evaluate these two identities (in the canonical representation) on the vectors $e_4=P_{22}e_4$ and $e_7=P_{11}e_7$ we get
\begin{align*}
& (f_1(U^2)+f_3(U^2)U)e_4=U^2 e_4\\
&(f_2(U^2)U^*+ f_4(U^2))e_7=U^{-2}e_7\; .
\end{align*} 
The vectors of the canonical basis are separating vectors for $C^*(U)$ and this implies that
\begin{align*}
& f_1(z^2)+f_3(z^2)z=z^2\\
& f_2(z^2)z^*+ f_4(z^2)=\bar z^2\; .
\end{align*} 
These equations immediately imply that $f_2(z)=f_3(z)=0$ (it is enough to notice that the functions $\{f_i(z^2)\}_{i=1}^4$ and $z^{\pm 2}$ are even), which in turn implies that $f_1(z)=z$, $f_4(z)=\bar z$. Therefore, we have that
$\tilde U= U^2 P_2+U^{-2}P_1$. However, by using the canonical representation it is easy to see that this operator does not verify the defining relations, as it does satisfy $\tilde U \tilde S_2=\tilde S_1$ but
it does not satisfy $\tilde U\tilde S_1= \tilde S_2\tilde U$.

\end{document}